\documentclass{article}

\usepackage{amssymb}
\usepackage{mathrsfs}
\usepackage{amsthm}

\begin{document}

\newtheorem{thm}{Theorem}[section]
\newtheorem{defn}[thm]{Definition}
\newtheorem{lem}[thm]{Lemma}
\newtheorem{prop}[thm]{Proposition}
\newtheorem{cor}[thm]{Corollary}
\newtheorem{rmk}[thm]{Remark}
\newtheorem{conv}[thm]{Convention}
\newtheorem{conj}[thm]{Conjecture}

\renewcommand{\qed}{\hfill \mbox{\raggedright \rule{.1in}{.1in}}}
\def\co{\colon\thinspace}

\title{On the coarse-geometric detection of subgroups}
\author{Diane M. Vavrichek}
\date{}
\maketitle

\abstract{We generalize \cite{Vav_qlines} to give sufficient conditions, primarily on coarse geometry, to ensure that a subset of a Cayley graph is a finite Hausdorff distance from a subgroup.
Using this result, we prove a partial converse to the Flat Torus Theorem for CAT(0) groups.
Also using this result, we give sufficient conditions for subgroups and splittings to be invariant under quasi-isometries.
}

\section{Introduction}

The Flat Torus Theorem is a well-known result in Geometric Group Theory, which deduces a geometric property of a metric space from an algebraic property of a group that acts nicely on that space.
It says that if $G$ is a group acting geometrically on a CAT(0) space $X$, and $G$ contains a free abelian subgroup of rank $n$, then $X$ must contain an $n$--flat (on which the subgroup acts cocompactly).
It is a well-known question going back to Gromov \cite[Section 6.B$_3$]{Gromov_asymptotic_invariants} whether the converse to this theorem holds, i.e. whether the existence of an $n$--flat in $X$ implies the existence of a free abelian subgroup of rank $n$ in $G$.
In the following we prove this converse, given some assumptions on the flat.
In addition, we show that, up to finite Hausdorff distances, the abelian subgroup we obtain acts cocompactly on the given flat, and give examples where this conclusion is false without our assumptions on the flat.

Specifically, if $G$ is a finitely generated group acting geometrically on a CAT(0) space $X$, and $F$ is an isometrically embedded copy of Euclidean space $\mathbb{E}^n$ in $X$ that satisfies three conditions on complementary components (including that the complement of a uniform neighborhood of $F$ has at least three components that are unbounded away from $F$), then we show that $G$ contains a subgroup isomorphic to $\mathbb{Z}^n$, with any orbit a finite Hausdorff distance from $F$ (Theorem \ref{flat_torus_converse}).
This conclusion is false if any one of our three hypotheses are removed.

In developing the tools needed to prove this theorem, we show that certain, mainly coarse-geometric conditions on a subset of the Cayley graph of a group imply that that subset is a finite Hausdorff distance from a subgroup (Theorem \ref{subgroup_thm}).
We go on to show that related hypotheses imply that certain types of subgroups and splittings are invariant under quasi-isometries (Theorem \ref{main_thm} and Corollary \ref{splitting_cor} respectively).
\\

We shall proceed by giving careful statements of these results.
Recall that a group action is said to be geometric if it is a properly discontinuous and cocompact action by isometries.
A geodesic space is said to be CAT(0) if geodesic triangles are ``thinner'' than their comparison triangles in the Euclidean plane (see, for instance, \cite{BridsonHaefliger}).
The Flat Torus Theorem states that if $G$ is a group acting geometrically on a CAT(0) space $X$, and $H \cong \mathbb{Z}^n$ is a subgroup of $G$, then $X$ contains an isometrically embedded copy of $\mathbb{E}^n$, on which $H$ acts with a torus quotient.  
(See \cite[Chapter II.7]{BridsonHaefliger}.)

In order to state our partial converse to this result, we must define a few terms.
Let $X$ be a metric space and let $Y,Z \subseteq X$.
We say that a neighborhood of $Z$ is uniform if it is an $r$--neighborhood of $Z$, for some $r \geq 0$.
For any $r \geq 0$, we say that a component of the complement of $N_r(Y)$ is shallow if it is contained in a uniform neighborhood of $N_r(Y)$, and we say it is deep otherwise.
$Y$ is said to satisfy the deep condition if, for every $r\geq 0$, there is some $m_0 \geq 0$ such that the $m_0$-neighborhood of each deep component of $(X-N_r(Y))$ contains $N_r(Y)$.
We say that $Y$ satisfies the shallow condition if, for every $r \geq 0$, there is some $m_1 \geq 0$ such that the $m_1$-neighborhood of $N_r(Y)$ contains all shallow components of the complement of $N_r(Y)$.
We say that $Y$ satisfies the 3--separating condition if it has at least three deep complementary components.
We use $d_{Haus}$ to denote Hausdorff distance.
\\

\noindent {\bf Theorem \ref{flat_torus_converse}. }{\it
Let $X$ be a CAT(0) space, and let $G$ be a finitely generated group acting geometrically on $X$.
Suppose that $X$ contains an isometrically embedded copy $F$ of $\mathbb{E}^n$ that has a uniform neighborhood that satisfies the deep, shallow and 3--separating conditions.

Then $G$ contains a subgroup $H \cong \mathbb{Z}^n$, such that for any $x_0 \in X$, 
$$d_{Haus}(Hx_0, F)<\infty.$$
}

\begin{rmk}\label{counterexs}
The conclusion in Theorem \ref{flat_torus_converse} is false without the deep, shallow and 3--separating hypotheses.
For example, consider the free group on $m$ generators, $F_m$, together with its action on its standard Cayley graph, $\mathscr{C}(F_m)$, which is a regular $2m$-valent tree.
Then any geodesic line in $\mathscr{C}(F_m)$ is an isometrically embedded copy of $\mathbb{R} \cong \mathbb{E}^1$, that satisfies the shallow and 3--separating conditions, but not the deep condition.
However, this line corresponds to a subgroup if and only if it is periodic with respect to its edge labels in the standard generating set.

The 3--separating hypothesis is also necessary in general.  
Consider, for instance, $\mathbb{Z}^2$, together with the standard action on $\mathbb{R}^2$.
Any line in $\mathbb{R}^2$ is an isometrically embedded copy of $\mathbb{E}^1$, and satisfies the deep and shallow conditions.
However, a line corresponds to a subgroup of $\mathbb{Z}^2$ if and only if it has rational slope.
(Note that both this and the previous example can be generalized to examples with isometrically embedded copies of $\mathbb{E}^n$ for any $n$, by taking products of the groups with $\mathbb{Z}^{n-1}$ and the spaces with $\mathbb{R}^{n-1}$.)

Finally, Lemma \ref{shallow_for_subgroups} shows that the conclusion of the theorem is necessarily false if the shallow condition is not satisfied.
\end{rmk}

Partial converses to the Flat Torus Theorem have appeared in \cite{Hruska2}, \cite{CapraceHaglund} and \cite{CapraceMonod_discrete}.
In Theorem 3.7 of \cite{Hruska2}, Hruska shows the converse to the Flat Torus Theorem, assuming that $X$ has an ``isolated flats'' property.
We note that Theorem \ref{flat_torus_converse} does not require that the flats of $X$ are isolated, or even that the copy of $\mathbb{E}^n$ is a maximal flat subspace of $X$.

The main result of \cite{CapraceHaglund} also overlaps with Theorem \ref{flat_torus_converse}.
In this paper, Caprace and Haglund show, in particular, that if $W$ is a Coxeter group and the Davis complex of $W$ contains an isometrically embedded copy of $\mathbb{E}^n$ then $W$ contains a subgroup that is isomorphic to $\mathbb{Z}^n$.

Also, Caprace and Monod prove in Theorem 3.8 of \cite{CapraceMonod_discrete} that if $X$ is a proper CAT(0) symmetric space with cocompact isometry group that acts minimally on $X$, and $X \cong \mathbb{R}^n \times X'$, then any lattice in $Isom(X)$ contains a $\mathbb{Z}^n$ subgroup that acts cocompactly on the Euclidean factor.
\\

We shall denote by $\mathscr{C}(G)$ the Cayley graph of a finitely generated group $G$.  
We say that a subset $Y \subseteq \mathscr{C}(G)$ satisfies the noncrossing condition if there is some $k>0$ such that, for all $g \in G$, $gY$ is contained in the $k$--neighborhood of a deep component of the complement of $Y$.

The following theorem about ``subgroup detection'' is the main ingredient used in the proof of Theorem \ref{flat_torus_converse}.
\\

\noindent {\bf Theorem \ref{subgroup_thm}. }{\it
Let $G$ be a finitely generated group, let $Y$ be a subgraph of $\mathscr{C}(G)$ that satisfies the deep, shallow, 3--separating and noncrossing conditions.  Then $G$ contains a subgroup $H$ such that 
$$d_{Haus}(Y,H)<\infty.$$
}

We note that in general, $H$ will be badly distorted in $G$.  In particular, we will not typically get that $H$ is quasi-isometrically embedded in $G$.
In Section \ref{algebraic_invariants}, we will see that we can still detect certain algebraic properties of $H$ from geometric properties of $Y$.  

The proof of Theorem \ref{subgroup_thm} makes crucial use of work of Papasoglu that appears in \cite{Papa_asym_top}.

In the case that $Y$ is a ``uniformly distorted'' copy of $\mathbb{R}$, it was shown in \cite{Papa_qlines} and \cite{Vav_qlines} that $Y$ satisfies the deep and noncrossing conditions as long as $G$ is finitely presented and one-ended.
Proposition 3.5 of \cite{Vav_qlines} gives the conclusion of Theorem \ref{subgroup_thm} in the special case that these assumptions hold.

We conjecture that the noncrossing condition often holds for connected, 3--separating subspaces; see Conjecture \ref{noncrossing_conj}.
However, Remark \ref{counterexs} essentially shows that the conclusion of Theorem \ref{subgroup_thm} is false without the deep, shallow and 3--separating assumptions.

We show in Sections \ref{terminology_section} and \ref{final_section} that, with the exception of the noncrossing condition, the hypotheses of Theorem \ref{subgroup_thm} are invariant under quasi-isometries, and correspond to certain algebraic conditions for subgroups.  
In addition, we discuss a theory of ``coarse isometries'' in Section \ref{ud_section} that allows us to conclude that finitely generated subgroups that correspond under a quasi-isometry, up to finite Hausdorff distance, are quasi-isometric themselves.
This allows us to conclude that, up to the satisfaction of the noncrossing condition, certain subgroups are preserved under quasi-isometries, in the following strong sense.
The notion of the number of coends of a subgroup is originally due to Kropholler and Roller, and is discussed in Section \ref{final_section}.
\\

\noindent {\bf Theorem \ref{main_thm}. }{\it
Let $G$ and $G'$ be finitely generated groups and let $f\co \mathscr{C}(G) \to \mathscr{C}(G')$ be a quasi-isometry.
Suppose that $H$ is a subgroup of $G$ that has at least three coends in $G$, and, for any infinite index subgroup $K$ of $H$, $K$ has only one coend in $G$.

If sufficiently large uniform neighborhoods of $f(H)$ in $\mathscr{C}(G')$ satisfy the noncrossing condition, then $G'$ contains a subgroup $H'$ such that 
$$d_{Haus}(H', f(H))<\infty.$$
In addition, $H$ is finitely generated if and only if $H'$ is finitely generated, and in this case $H$ is quasi-isometric to $H'$.
}\\

This theorem generalizes the following result from \cite{Vav_qlines}.

\begin{thm}\label{qlines1}
{\rm {\bf \cite{Vav_qlines}}}
Let $f \co \mathscr{C}(G) \to \mathscr{C}(G')$ be a quasi-isometry between one-ended, finitely presented groups, and suppose that $G$ contains a two-ended subgroup $H$ that has at least three coends in $G$.
Then there is a subgroup $H' \cong \mathbb{Z}$ of $G'$ such that
$$d_{Haus}(H', f(H))<\infty.$$
\end{thm}

In the setting above, the noncrossing condition is satisfied if $G$ (hence $G'$) is finitely presented and one-ended, and the hypothesis about subgroups of $H$ is equivalent to $G$ being one-ended.

In general, Theorem \ref{main_thm} is false without the hypotheses on coends.
To see this, consider the examples given in Remark \ref{counterexs}.  
If $\gamma$ and $\gamma'$ denote geodesics in the tree $\mathscr{C}(F_2)$ that are periodic and aperiodic respectively, with respect to edge labelings, then note that there is an isometry $f \co \mathscr{C}(F_m) \to \mathscr{C}(F_m)$ that interchanges $\gamma$ and $\gamma'$.
The line $\gamma$ is a finite Hausdorff distance from an infinite cyclic subgroup $H$ of $F_m$, and $f(\gamma) = \gamma'$ is an infinite Hausdorff distance from any subgroup of $F_m$.
However, $F_m$ is not one-ended, hence the hypothesis in Theorem \ref{main_thm} about infinite index subgroups of $H$ is not satisfied.

Also consider $\mathscr{C}(\mathbb{Z}^2)$ embedded in $\mathbb{R}^2$ in the standard way, and let $f \co \mathscr{C}(\mathbb{Z}^2) \to \mathscr{C}(\mathbb{Z}^2)$ be a quasi-isometry given by rotation by an irrational angle, composed with nearest point projection back into $\mathscr{C}(\mathbb{Z}^2)$.
If $H$ is any infinite cyclic subgroup of $\mathbb{Z}^2$, then $H$ is a finite Hausdorff distance from a line in $\mathbb{R}^2$ with rational slope, hence $f(H)$ is a finite Hausdorff distance from a line with irrational slope.
Thus $f(H)$ is an infinite Hausdorff distance from any subgroup of $\mathbb{Z}^2$.

Both of these counterexamples can be generalized to ones, for instance, where $H$ is isomorphic to $\mathbb{Z}^n$ for any $n$.

Using the work of Dunwoody and Swenson \cite{DunSwen}, we show in Theorem \ref{splitting_thm} that we can choose the subgroup $H$ in Theorem \ref{subgroup_thm} to be such that $G$ splits over $H$, as an amalgamated free product or HNN extension.
This splitting will ``have three coends'' (meaning that $H$ has three coends in $G$), which allows us to draw some algebraic conclusions about the splitting in Theorem \ref{etilde_equals_2}.
Finally, combining Theorem \ref{splitting_thm} with Theorem \ref{main_thm}, we get that these types of splittings are invariant under quasi-isometries, in the following setting.
\\

\noindent {\bf Corollary \ref{splitting_cor}. }{\it
Let $G$ and $G'$ be finitely generated groups and let $f\co \mathscr{C}(G) \to \mathscr{C}(G')$ be a quasi-isometry.
Suppose that $H$ is a finitely generated subgroup of $G$ such that for any infinite index subgroup $K$ of $H$, $K$ has one coend in $G$.
Suppose also that $G$ admits a splitting over $H$ that has three coends.

If sufficiently large uniform neighborhoods of $f(H)$ in $\mathscr{C}(G')$ satisfy the noncrossing condition, then $G'$ contains a finitely generated subgroup $H'$ such that $H'$ is quasi-isometric to $H$, $d_{Haus}(H', f(H))<\infty,$ and $G'$ admits a splitting over $H'$ that has three coends.
}\\

The outline of this paper is as follows.
In Sections \ref{terminology_section} and \ref{ud_section}, we will give most of our definitions and some basic facts.
The deep, shallow, 3--separating and noncrossing conditions will be discussed in Section \ref{terminology_section}, and uniformly distorting maps and coarse isometries will be defined in Section \ref{ud_section}.

In Section \ref{subgroups_section}, we will prove that any subset of a Cayley graph that satisfies the four conditions mentioned is a finite Hausdorff distance from a subgroup.
In Section \ref{noncrossing_section}, we will give our partial converse to the Flat Torus Theorem.

In Section \ref{algebraic_invariants}, we will discuss some coarse isometry invariants, and see how they show that the coarse geometry of the subset from Section \ref{subgroups_section} has implications about the algebra of the nearby subgroup.

In Section \ref{final_section}, we will see how our subgroup detection theorem can be applied to show the quasi-isometry invariance of certain subgroups, and we will also show the invariance of the associated commensurizer subgroups, given an assumption that involves the noncrossing condition.
Finally, in Section \ref{splittings_section}, we will see that we get group splittings in the settings in which we are working.
\\

 \noindent {\bf Acknowledgments.}
 The author is grateful Peter Scott, for many valuable discussions and comments.
 The author also thanks Mario Bonk, Matt Brin, and Panos Papasoglu for several helpful conversations.
 In addition, the author gratefully acknowledges that this research was partially supported by the University of Strasbourg, and NSF grant DMS-0602191.

\section{Preliminaries}\label{terminology_section}

In this section, we will state our conventions, define some necessary terminology, and prove a few basic facts related to some of our nonstandard terms.


We shall assume throughout this paper that balls are open and neighborhoods closed, i.e. for any metric space $X$, 
$$B_r(p) = \{ x \in X : d(p,x)<r\}, \ \  N_r(Y) = \{ x \in X : d(x,Y) \leq r\}$$
 for all $p \in X, Y \subseteq X$ and $r \geq 0$.
 The $r$--neighborhood of $Y$ refers to $N_r(Y)$.
 
 \begin{defn}
 If $X$ is a metric space and $Y \subseteq X$, then we say that a neighborhood $N$ of $Y$ is a {\em uniform neighborhood} if $N = N_r(Y)$ for some $r \geq 0$.
\end{defn}

If $X$ and $Y$ are metric spaces, $\Lambda \geq 1$ and $C \geq 0$, then a map $f \co X \to Y$ is a $(\Lambda, C)$ quasi-isometric embedding if, for all $x,x' \in X$,
$$\frac{1}{\Lambda}d(x,x')-C \leq d(f(x), f(x')) \leq \Lambda d(x,x') +C.$$

\begin{defn}
If $f \co X \to Y$ is a map between metric spaces and $C \geq 0$, then we say that $f$ is {\em $C$--onto} if the $C$--neighborhood of $Im(f)$ in $Y$ is equal to $Y$.
We say that $f$ is {\em coarsely onto} if $f$ is $C$--onto for some $C$.
\end{defn}

We say that $f$ is a $(\Lambda, C)$ quasi-isometry if $f$ is a $(\Lambda, C)$ quasi-isometric embedding and is $C$--onto.
In this case, we call $\Lambda$ and $C$ the parameters of $f$.
We will say that $f$ is a quasi-isometric embedding (quasi-isometry respectively) if $f$ is a $(\Lambda, C)$ quasi-isometric embedding ($(\Lambda , C)$ quasi-isometry respectively) for some $\Lambda \geq 1$ and $C \geq 0$.

We say that a function $f_1 \co X \to Y$ has finite distance from a function $f_2 \co X \to Y$ if
$$d(f_1, f_2) := \sup_{x \in X} d_Y(f_1(x), f_2(x)) < \infty.$$
If $f \co X \to Y$ is a quasi-isometry, and $f' \co Y \to X$ is also a quasi-isometry such that both compositions $f'f$ and $ff'$ are a finite distance from the identity maps $id_X$ and $id_Y$ respectively, then we say that $f'$ is a quasi inverse to $f$.
It is a fact that any quasi-isometry $f$ has a quasi inverse $f'$, such that the parameters of $f'$ depend only on the parameters of $f$.

We will take all graphs to be metric graphs, with each edge of length one.

\begin{conv}
We assume that every finitely generated group mentioned in the following comes equipped with a fixed finite generating set.
\end{conv}

If $G$ is a finitely generated group, we denote by $\mathscr{C}(G)$ the associated Cayley graph.
Thus $\mathscr{C}(G)$ has vertex set equal to $G$, and $g,g' \in G$ span an edge if and only if $g' = gs$, where $s$ or $s^{-1}$ is an element of the generating set associated to $G$.
We use $d_G$ to denote metric on $\mathscr{C}(G)$, which restricts to the word length metric on $G$.
Note that $G$ acts on $\mathscr{C}(G)$ by isometries on the left.

We recall that any two finite generating sets for $G$ will yield quasi-isometric Cayley graphs, so the geometry of $G$ is uniquely determined up to quasi-isometry.
\\


Next, we will introduce many of the nonstandard terms that will be needed later.  
(We note that this terminology differs from that of \cite{Vav_qlines}:  `deep', `shallow' and `$n$--separating' replace `essential', `inessential' and `$n$--parting' respectively, and $deep(m_0)$ is a stronger condition than $ess(m_0)$.  

\begin{defn}
Let $X$ be a metric space and let $Y$ and $Z$ be subsets of $X$.  A component of $(X-Z)$ is {\em shallow} if it is contained in some uniform neighborhood of $Z$.
Otherwise, we say that the component is {\em deep}.

If $m_1 \co \mathbb{R}_{\geq 0} \to \mathbb{R}_{\geq 0}$ is such that for each $r\geq 0$, each shallow component of $(X-N_r(Y))$ is contained in the $m_1(r)$--neighborhood of $N_r(Y)$, then we say that $Y$ satisfies $shallow(m_1)$.

If $m_0 \co \mathbb{R}_{\geq 0} \to \mathbb{R}_{\geq 0}$ is such that for each $r \geq 0$ and each point $p \in N_r(Y)$, the ball $B_{m_0(r)}(p)$ meets each deep component of $(X-N_r(Y))$, then we say that $Y$ satisfies $deep(m_0)$.
Note that $m_0(r) \geq r$ for all $r$.

We say that $Y$ satisfies the {\em shallow} ({\em deep} respectively) {\em condition} if $Y$ satisfies $shallow(m_1)$ ($deep(m_0)$ respectively) for some $m_1$ ($m_0$ respectively).
\end{defn}

\begin{defn}
Let $X$ and $Y$ be as above, let $n >0$, and we shall say that $Y$ is {\em $n$--separating}, or satisfies the {\em $n$--separating condition}, if $(X-Y)$ has at least $n$ deep components.
\end{defn}

\begin{rmk}
We will be primarily interested in subspaces $Y$ when $X$ is the Cayley graph of a finitely generated group, say $\mathscr{C}(G)$.
Note that, as $G$ acts on $\mathscr{C}(G)$ by isometries, if $Y$ satisfies $deep(m_0)$, $shallow(m_1)$, or the $n$--separating condition, then so does any translate $gY$ of $Y$.
\end{rmk}

\begin{defn}
A set of subsets of a metric space $X$, $\mathscr{Y} = \{ Y_1, Y_2, \ldots \}$, is said to satisfy $noncrossing(k)$ if, for each $i \neq j$, $Y_i$ is contained in the $k$--neighborhood of some 
deep
component of $X-Y_j$.

Suppose a group $G$ acts on $X$, let $Y \subseteq X$ and let $k>0$.
We say that $Y$ satisfies $noncrossing(k)$ if $\{ gY\}_{g \in G}$ satisfies $noncrossing(k)$ in the previous sense.

We say that $Y$ ($\mathscr{Y}$ respectively) satisfies the {\em noncrossing condition} if $Y$ ($\mathscr{Y}$ respectively) satisfies $noncrossing(k)$ for some $k$.
\end{defn}

We will see in Section \ref{final_section} that there are interesting situations when subsets of Cayley graphs naturally satisfy many of these conditions.
\\

As we will show next, the deep, shallow and $n$--separating conditions are invariant under quasi-isometries, in a suitable sense.  

The proof of the next lemma follows an analogous argument in \cite{Vav_qlines}.

\begin{lem}\label{shallow_qi_inv}
Let $f \co X \to X'$ be a quasi-isometry between geodesic spaces, with $Y \subseteq X$ and $Y' \subseteq X'$ such that $d_{Haus}(Y', f(Y))<\infty$.

If $Y$ satisfies the shallow condition then $Y'$ also satisfies the shallow condition.
\end{lem}

\begin{proof}
Suppose that $Y$ satisfies $shallow(m_1)$, let $Y''$ denote $N_r(Y')$ for some $r \geq 0$, and we shall show that there is some constant $m_1'(r)$ such that any shallow component of the complement of $Y''$ is contained in the $m_1'(r)$--neighborhood of $Y''$.

Let $f'$ denote a quasi-inverse to $f$.
We claim that there is some $R>0$ that is large enough so that if $C, C'$ are distinct components of $(X'-Y'')$, then no component of $(X-N_R(Y))$ meets both $f'(C)$ and $f'(C')$.
To see this, let $\Lambda', K'$ be such that $f'$ is a $(\Lambda', K')$ quasi-isometry, let $R' > \Lambda'K'$ and let $R\geq 0$ be such that $f'(N_{R'}(Y'')) \subseteq N_R(Y)$.  
Note that any points in distinct components of the complement of $Y''$ that are also in the complement of $N_{R'}(Y'')$ are a distance of more than $2R'$, and hence more than $2\Lambda' K'$, from one another.

Since $f'$ is $K'$--onto and  $f'(N_{R'}(Y'')) \subseteq N_R(Y)$, if some component of $(X-N_R(Y))$ met the image under $f'$ of more than one component of $(X'-Y'')$, then there must be distinct components $C,C'$ of the complement of $Y''$ and points $p \in C, p' \in C'$ such that $d(f'(p), f'(p')) \leq K'$.  
But $(1/\Lambda')d(p,p')-K' \leq d(f'(p), f'(p'))$, since $f'$ is a $(\Lambda', K')$ quasi-isometry, and $2\Lambda'K' < d(p,p')$, by the previous paragraph.
Combining these equations yields $K'<d(f'(p), f'(p'))$, a contradiction. 
Thus each component of the complement of $N_R(Y)$ is met by the image under $f'$ of no more than one component of the complement of $Y''$.

Since $f'$ is coarsely onto, it follows that any deep component of the complement of $N_R(Y)$ is met only by the image under $f'$ of a deep component of the complement of $Y''$.
Hence, for any shallow component $S$ of the complement of $Y''$, $f'(S)$ is contained in the union of $N_R(Y)$ with shallow components of the complement of $N_R(Y)$, and hence is contained in the $m_1(R)$--neighborhood of $N_R(Y)$.
It follows that there is a constant $m_1'(r)$, depending only on $f$, $f'$, $R$, $m_1(R)$ and $d_{Haus}(Y'', f(Y))$, such that $S$ is contained in the $m_1'(r)$-neighborhood of $Y''$.
\end{proof}

With slight alteration, and in light of Lemma \ref{shallow_qi_inv}, \cite{Vav_qlines} shows that $n$--separation is a quasi-isometry invariant:

\begin{lem}\label{n--sep_qi_inv}
Let $f \co X \to X'$ be a quasi-isometry between geodesic spaces, with $Y \subseteq X$ and $Y' \subseteq X'$ such that $Y$ and $Y'$ both satisfy the shallow condition and $d_{Haus}(f(Y), Y')<\infty$.

Then there is some $R \geq 0$ such that for any $n>0$, if $Y$ is $n$--separating then $N_{R}(Y')$ is $n$--separating.
\end{lem}

Our next result, which is about quasi-isometries and the deep condition, also requires that the shallow condition is satisfied.

\begin{lem}\label{deep_qi_inv}
Let $f \co X \to X'$ be a quasi-isometry between geodesic spaces, with $Y \subseteq X$ and $Y' \subseteq X'$ such that both $Y$ and $Y'$ satisfy the shallow condition and $d_{Haus}(f(Y), Y')<\infty$.

If $Y$ satisfies the deep condition, then $Y'$ also satisfies the deep condition.
\end{lem}

\begin{proof}
Suppose that $Y$ satisfies $deep(m_0)$, and let $f'$ be a quasi-inverse to $f$.
Suppose that $f$ is a $(\Lambda, K)$ quasi-isometry and that $f'$ is a $(\Lambda', K')$ quasi-isometry.
Fix $r \geq 0$, let $Y'' = N_r(Y')$, and we will show that there is some constant $m_0'(r)$ such that $B_{m_0'(r)}(p')$ meets all deep components of the complement of $Y''$, for any $p' \in Y''$.

As we saw in the proof of Lemma \ref{shallow_qi_inv}, there is some $R \geq 0$ such that $f'(Y'') \subseteq N_R(Y)$, and each component of the complement of $N_R(Y)$ meets the image under $f'$ of at most one component of $(X'-Y'')$.
Let $D'$ be a deep component of $(X'-Y'')$.
As $Y$ satisfies the shallow condition, $f'(D')$ must meet a deep component of $(X-N_R(Y))$, say $D$.
Since $f'$ is $K'$--onto, by enlarging $R$ by $K'$ if necessary, we can assume that $D \subseteq N_{K'}(f'(D'))$.

Thus $B_{m_0(R)+K'}(p)$ meets $f'(D')$ for any $p \in N_R(Y)$.
In particular, fix any $p' \i nY''$ and take $p = f'(p')$, and there is some $d' \in D'$ such that $d(f'(p'), f'(d'))<(m_0(R)+K')$, and hence $d(ff'(p'), ff'(d'))< \Lambda (m_0(R)+K')+K$.
The distance from $ff'$ to the identity on $X'$ is bounded by a function of $\Lambda $ and $K$ and hence 
$$d(p', D') \leq d(p', d') $$
$$\leq d(p', ff'(p')) + d(ff'(p'), ff'(d')) + d(ff'(d'), d')$$
$$\leq (\Lambda (m_0(R)+K')+K) + 2d(ff', id_{X'}).$$
Thus our claim follows, for any $m_0'(r) > \Lambda (m_0(R)+K')+K + 2d(ff', id_{X'})$.
It follows that $Y'$ satisfies the deep condition.
\end{proof}

\section{Uniformly distorting maps and coarse isometries}\label{ud_section}

Next, we will introduce uniformly distorting maps and coarse isometries, and make some geometric observations about subgroups.
Coarse isometries will provide a useful generalization of quasi-isometries, and we will see how both of these types of functions arise naturally when considering subgroups in Cayley graphs.
Moreover, we will see in Proposition \ref{ci_is_qi} that in many of the situations that we are concerned with, coarse isometries are in fact quasi-isometries.

\begin{defn}\label{ud_defn}
Let $(X,d_X)$ and $(Y,d_Y)$ be metric spaces, and let $\phi$ and $\Phi$ be weakly increasing proper functions from $\mathbb{R}_{\geq 0}$ to $\mathbb{R}_{\geq 0}$.
Then we shall say that a function $f \co X \to Y$ is a {\em $(\phi, \Phi)$--uniformly distorting map} if, for any $x, x' \in X$ and $r \in \mathbb{R}_{\geq 0}$,
\begin{enumerate}
\item{if $d_X(x,x') \geq r$ then $d_Y(f(x), f(x')) \geq \phi (r)$, and}
\item{if $d_X(x,x') \leq r$ then $d_Y(f(x), f(x')) \leq \Phi(r)$.}
\end{enumerate}
Thus $\phi (d_X(x,x')) \leq d_Y(f(x), f(x')) \leq \Phi ( d_X(x,x'))$.

We shall say that $f$ is a {\em uniformly distorting map} if $f$ is $(\phi, \Phi)$--uniformly distorting for some $\phi$ and $\Phi$.
\end{defn}

Note that we do not require a uniformly distorting map to be continuous.
Note also that the composition of uniformly distorting maps is uniformly distorting.

\begin{defn}
If $f$ is both uniformly distorting and coarsely onto, then we say that $f$ is a {\em coarse isometry}.
\end{defn}

Note that any quasi-isometry is a coarse isometry.
In particular, if $Y_1, Y_2 \subseteq X$ are such that $d_{Haus}(Y_1, Y_2)<\infty$, then a nearest point projection map of $Y_1$ onto $Y_2$ is a coarse isometry.

In analogy to quasi inverses, we have the following.

\begin{lem}\label{coarse_inverses}{\em {\bf \cite{Vav_qlines}}}
If $f \co X \to Y$ is a coarse isometry between metric spaces, then there is a coarse isometry $f' \co Y \to X$ such that $f'f$ and $ff'$ have finite distances from the identity functions $id_X$ and $id_Y$ respectively.
\end{lem}

\begin{defn}
We shall call any function $f'$ satisfying the conclusion of the above lemma a {\em coarse inverse} to $f$.
\end{defn}

Also we note the following observation.

\begin{lem}
Suppose that $f \co X \to Y, g \co Y \to Z$ are coarse isometries.  Then $g f \co X \to Z$ is also a coarse isometry.
\end{lem}

In the next result, we make the key observation that often coarse isometries actually are quasi-isometries.

\begin{prop}\label{ci_is_qi}
Let $(X, d_X)$ and $(Y,d_Y)$ be geodesic metric spaces.
Then any coarse isometry between them is a quasi-isometry.
\end{prop}

\begin{proof}
Let $f \co X \to Y$ be a $(\phi, \Phi)$--uniformly distorting map that is a coarse isometry, and let $f'$ be a coarse inverse to $f$.
We will argue that there is some $\Lambda_1 \geq 1, C_1\geq 0$ depending on $\phi$ and $\Phi$ such that, for all $p_1, p_2 \in X$,
$$d_Y(f(p_1), f(p_2)) \leq \Lambda_1 d_X(p_1, p_2) + C_1.$$

This will imply the proposition, for we can run this argument with $f'$ replacing $f$ to get constants $\Lambda_2, C_2$ such that for all $q_1, q_2 \in Y$,
$$d_X(f'(q_1), f'(q_2)) \leq \Lambda_2 d_Y(q_1, q_2)+C_2.$$
As $d(id_X, f'f)<\infty$, we have $d_X(p_1, p_2) \leq d_X(f'f(p_1), f'f(p_2)) + 2d(id_X, f'f)$, and combining this with the above equation, taking $f(p_i)$ for $q_i$, gives
$$d_X(p_1, p_2) \leq \Lambda_2 d_Y(f(p_1), f(p_2)) + (C_2 + 2d(id_X,f'f)).$$
Hence 
$$(1/\Lambda_2)d_X(p_1, p_2) - (1/\Lambda_2)(C_2 + 2d(id_X, f'f))$$
$$ \leq d_Y(f(p_1), f(p_2)) \leq \Lambda_1 d_X(p_1, p_2) + C_1,$$
so $f$ is a quasi-isometry.
\\

It remains to show that we can find $\Lambda_1\geq 1, C_1\geq 0$ so that, for all $p_1, p_2 \in X$,
$$d_Y(f(p_1), f(p_2)) \leq \Lambda_1 d_X(p_1, p_2) + C_1.$$
Fix $\epsilon>0$ and let $\epsilon' = \Phi(\epsilon)$.
Then for any $p_1, p_2 \in X$, there is a sequence of points $x_0 = p_1, x_2, x_3, \ldots , x_k$ along a geodesic from $p_1$ to $p_2$ such that $d_X(x_i, x_{i+1}) = \epsilon$ for all $i<k$, $d_X(x_k, p_2) \leq \epsilon$, and
$$\sum_{i=0}^{k-1}d_X(x_i, x_{i+1})=k\epsilon \leq d_X(p_1, p_2).$$
Hence
$$k \leq d_X(p_1, p_2)/\epsilon.$$

Note that $\{ f(x_i)\}$ is a sequence of points from $f(p_1)$ to $f(p_2)$, such that for each $i<k$, 
$d_Y(f(x_i), f(x_{i+1})) \leq \epsilon'$
and $d_Y(f(x_k), f(p_2)) \leq \epsilon'$.
Hence 
$$d_Y(f(p_1), f(p_2)) \leq \sum_{i=0}^{k-1}d_Y(f(x_i), f(x_{i+1})) + d_Y(f(x_k), f(p_2)) \leq \epsilon'(k+1)$$
$$ \leq \epsilon'[d_X(p_1, p_2)/\epsilon + 1] = \frac{\epsilon'}{\epsilon}d_X(p_1, p_2) + \epsilon'.$$ 
Thus our claim follows, for $\Lambda_1 =  \max \{ 1, \frac{\epsilon'}{\epsilon} \}$ and $C_1 = \epsilon'$.
\end{proof}

Let $H$ be a finitely generated subgroup of a finitely generated group $G$.
Then we can consider $H$ with respect to its own intrinsic geometry, $(H, d_H)$, or with respect to the geometric structure induced by $G$, $(H, d_G)$.
Our main interest in coarse isometries stems from the fact that these two spaces are coarsely isometric:

\begin{lem}\label{inclusion_is_wup} {\em {\bf \cite{Vav_qlines}}}
Let $G$ be a finitely generated group with finitely generated subgroup $H$.  Then the inclusion map $i_H \co (H, d_H) \to (H, d_G)$ is uniformly distorting, hence is a coarse isometry.

Moreover, the bound on expansion can be taken to be linear.
That is, $i_H$ is $(\phi, \Phi)$--uniformly distorting, for some $\phi$ and $\Phi$, where we can take $\Phi(r) \leq Lr$ for all $r$ and some constant $L >0$.
\end{lem}

Finally, we will need to understand from coarse geometry when subgroups of finitely generated groups are finitely generated themselves.
To this end, we introduce the following.
\begin{defn}\label{0-conn_defn}
We shall say that a subset $Z \subseteq \mathscr{C}(G)$ is {\em coarsely 0--connected} if there is some $r \geq 0$ such that $N_r(Z)$ is connected.
\end{defn}

\cite{Vav_qlines} implies the next fact.

\begin{prop}\label{subgroup_fg}
Let $G$ be a finitely generated group and let $H$ be a subgroup of $G$.
Then $H$ is finitely generated if and only if $H$ is coarsely 0--connected, as a subset of $\mathscr{C}(G)$.
\end{prop}

Suppose that there is some subset $Y \subseteq \mathscr{C}(G)$ such that $d_{Haus}(Y,H)<\infty$, and note that $H$ is 0--connected if $Y$ is connected.
Hence Proposition \ref{subgroup_fg} implies the following.

\begin{cor}\label{Y_conn_implies_H_fg}
Let $G$ be a finitely generated group with $H$ a subgroup of $G$ and $Y$ a connected subset of $\mathscr{C}(G)$.
If $d_{Haus}(Y,H)<\infty$ then $H$ is finitely generated.
\end{cor}

\section{Detecting subgroups}\label{subgroups_section}

In this section, we will show that subsets of Cayley graphs that satisfy the deep, shallow, 3--separating and noncrossing conditions are, up to a finite Hausdorff distance, subgroups of the ambient group.
First we will need the following two lemmas.

\begin{lem}\label{cis}
Suppose that $Y, Y' \subseteq \mathscr{C}(G)$ are 2--separating and satisfy $deep(m_0)$ and the shallow condition, and that $Y' \subseteq N_r(Y)$.
Then there is some constant that we will denote by $r_1(r)$, but which depends on $r$ and $m_0$, such that $r_1(r)>r$ and $Y\subseteq N_{r_1(r)}(Y')$.
\end{lem}

\begin{proof}
We will show the lemma for $r_1(r)=[2m_0(r+m_0(r))+m_0(r)]$.

Suppose that there are there are two deep components of the complement of $Y'$, $D_1$ and $D_2$, such that $Y$ meets each $D_i$ in a point $p_i$ that is not contained in the $m_0(r)$-neighborhood of $Y'$.
Thus, for each $i$, $B_{m_0(r)}(p_i) \subseteq D_i$.
Since $Y' \subseteq N_r(Y)$, the components of the complement of $N_r(Y)$ are contained in the components of the complement of $Y'$.
In particular, any deep component of the complement of $N_r(Y)$ is disjoint from at least one of $D_1$ or $D_2$.
But then this component must be disjoint from $B_{m_0(r)}(p_1)$ or from $B_{m_0(r)}(p_2)$, contradicting that $Y$ satisfies $deep(m_0)$.

Hence, since $Y'$ is 2--separating, there is a deep component $D'$ of the complement of $Y'$ such that $(Y \cap D') \subseteq N_{m_0(r)}(Y')$.
Since $D'$ is not contained in any uniform neighborhood of $Y'$, it follows that it is also not contained in any uniform neighborhood of $Y$.
Since $N_{m_0(r)}(Y') \subseteq N_{r+m_0(r)}(Y)$, and $Y$ satisfies the shallow condition, there must be a deep component $D$ of the complement of $N_{r+m_0(r)}(Y)$ that is contained in $D'$.
Let $fr(D)$ denote the frontier of $D$, and note that since $D$ is deep, $d_{Haus}(fr(D), Y) \leq m_0(r+m_0(r))$.

For any point $p \in D'$, consider a shortest path from $p$ to $Y$.
If this path meets $Y'$, then $d(p,Y') \leq d(p,Y)$.
Otherwise, the path is entirely contained in $D'$, so its endpoint is in $(Y\cap D')$.
Recall that $(Y \cap D') \subseteq N_{m_0(r)}(Y')$, and it follows that $d(p,Y') \leq (d(p, Y)+m_0(r))$.
Hence in either case, $d(p,Y') \leq (d(p, Y)+m_0(r))$.
Since $fr(D) \subseteq N_{m_0(r+m_0(r))}(Y)$, and also $fr(D) \subseteq D'$, it follows that $fr(D)$ is contained in the $[m_0(r+m_0(r))+m_0(r)]$-neighborhood of $Y'$.  
Since $Y \subset N_{m_0(r+m_0(r))}(fr(D))$, it follows that $Y$ is contained in the $[2m_0(r+m_0(r))+m_0(r)]$-neighborhood of $Y'$, as desired.
\end{proof}

The next result is a slight generalization of a lemma from \cite{Vav_qlines}.

\begin{lem} \label{seq_is_fin}
Let $G$ be a finitely generated group and let $\mathscr{Y}$ be a collection of 3--separating subsets of $\mathscr{C}(G)$ that satisfies $noncrossing(k)$, and assume that each $Y \in \mathscr{Y}$ satisfies $deep(m_0)$ and $shallow(m_1)$.
Moreover, suppose that there is some ball $B_s(v)$ in $\mathscr{C}(G)$ that meets each $Y \in \mathscr{Y}$.

Then there is a constant $x_1$ (which is independent of $s$) such that if, for all distinct $Y, Y' \in \mathscr{Y}$, $d_{Haus}(Y ,Y') > x_1$, then $\mathscr{Y}$ is finite.
\end{lem}

\begin{proof}
By Lemma \ref{cis}, for every $r \geq 0$ there is a constant $r_1(r) > r$ such that for all $Y, Y' \in \mathscr{Y}$, $Y' \subseteq N_r(Y)$ implies $d_{Haus}(Y,Y') \leq r_1(r)$.
Let $x_2 > m_0(k)$ and let $x_1 = r_1(x_2)$.
Then suppose both that $d_{Haus}(Y ,Y') > x_1$ for all $Y, Y' \in \mathscr{Y}$, and that $\mathscr{Y}$ is infinite.
Hence for any $Y, Y' \in \mathscr{Y}$, $Y \nsubseteq N_{x_2}(Y')$.

Choose any $Y_0 \in \mathscr{Y}$.
As $Y_0$ satisfies $deep(m_0)$ and $\mathscr{C}(G)$ is locally finite, the complement of $Y_0$ has only finitely many deep components.
As $\mathscr{Y}$ satisfies $noncrossing(k)$, there must be some deep component $C_0$ whose $k$--neighborhood contains infinitely many elements of $\mathscr{Y}$.

Let $\mathscr{Y}_1 = \{ Y \in [\mathscr{Y}-\{Y_0\}] : Y \subseteq N_k(C_0)\}$.
Choose $Y_1$ from $\mathscr{Y}_1$, and let $C_1'$ be the deep component of the complement of $Y_1$ whose $k$--neighborhood contains $Y_0$.
As $\mathscr{Y}_1$ is infinite, there is some deep component of the complement of $Y_1$ whose $k$--neighborhood contains infinitely many elements of $\mathscr{Y}_1$.
Let $C_1$ denote this component, and let $\mathscr{Y}_2 = \{ Y \in [\mathscr{Y}_1-\{ Y_0, Y_1\} ] : Y \subseteq N_k(C_1)\}$.
Choose $Y_2$ from $\mathscr{Y}_2$, and continue on in this manner.

This process produces an infinite sequence of elements of $\mathscr{Y}$, $\{ Y_i\}$, and subsets of $\mathscr{C}(G)$, $\{ C_i\}$ and $\{ C_i'\}$, such that, for each $i$, $C_i$ is a deep component of the complement of $Y_i$ such that $Y_j \subseteq N_k(C_i)$ for all $j>i$, and $C_i'$ is a deep component of the complement of $Y_i$ with $Y_j \subseteq N_k(C_i)$ for all $j<i$ (with perhaps $C_i = C_i'$).
Each $Y_i$ is 3--separating, so we may set $D_i$ to be a deep component of the complement of $Y_i$ that is not equal to $C_i$ or $C_i'$.

We will see next that the $D_i$'s are essentially disjoint.
We have that $(D_i-N_k(Y_i))$ is a collection of deep and shallow components of the complement of $N_k(Y_i)$.
Since $D_i$ is a deep component of the complement of $Y_i$ and $Y_i$ satisfies $shallow(m_1)$, it follows that $(D_i-N_k(Y_i))$ must contain a deep component of the complement of $N_k(Y_i)$, say $E_i$.

Now fix $i$ and $j$ to be distinct.
Since $Y_i$ is not contained in the $x_2$--neighborhood of $Y_j$, there must be some point $p \in Y_i$ such that $B_{x_2}(p)$ does not intersect $Y_j$.
As $x_2>m_0(k)\geq k$, $B_{x_2}(p)$ is contained in $C_j$ or $C_j'$.
Also $B_{x_2}(p)$ meets each deep component of the complement of $N_k(Y_i)$, and hence $(B_{x_2}(p) \cup E_i)$ is connected.

As $Y_j$ is contained in the $k$--neighborhood of $C_i$ or $C_i'$, we have that it is disjoint from $(D_i-N_k(Y_i))$, hence $Y_j$ does not meet $E_i$, or the union $(B_{x_2}(p) \cup E_i)$.
It follows that this union is contained in $C_j$ or $C_j'$, so is disjoint from $D_j$, and hence from $E_j \subseteq D_j$.
Thus, the $E_i$'s are disjoint.

Now we recall that all $Y \in \mathscr{Y}$ meet the ball $B_s(v)$, and hence $B_{s+m_0(k)}$ meets each $E_i$.
But there are infinitely many $E_i$'s, which we now know to be disjoint, while $\mathscr{C}(G)$ is locally finite, so we have reached a contradiction.
\end{proof}

The next theorem is one of our main results.  It makes use of an argument given in \cite{Papa_asym_top}.

\begin{thm}\label{subgroup_thm}
Let $G$ be a finitely generated group, let $Y$ be a subgraph of $\mathscr{C}(G)$ that satisfies the deep, shallow, 3--separating and noncrossing conditions.
Then $G$ contains a subgroup $H$ such that 
$$d_{Haus}(Y,H) < \infty.$$
\end{thm}

\begin{proof}
Suppose that $Y$ is as in the statement of the theorem.
We can assume that $Y$ contains $e \in G$ and is an infinite subgraph of $\mathscr{C}(G)$.
Let $Y$ satisfy $deep(m_0)$ and $noncrossing(k)$.

Let $\mathscr{Y} = \{ gY : gY$ meets the closed ball $\overline{B}_{k}(e) \}$.
Note that $\mathscr{Y}$ satisfies the hypotheses of Lemma \ref{seq_is_fin}; let $x_1$ be the constant from that lemma.
Consider two elements of $\mathscr{Y}$ equivalent if they are of finite Hausdorff distance from each other, and let $\{ \mathscr{Y}_i\}$ be the collection of equivalence classes of $\mathscr{Y}$.  

First, we will show the following claim:
\\

\noindent (*) $\mathscr{Y}$ is made up of only finitely many equivalence classes $\mathscr{Y}_i$, and, for each $i$ and each $gY \in \mathscr{Y}_i$ we have that
$$\sup_{g'Y \in \mathscr{Y}_i} d_{Haus}(gY, g'Y)$$
is finite.
\\

For suppose that either there are infinitely many equivalence classes $\mathscr{Y}_i$, or that there is some $i$ and some $gY \in \mathscr{Y}_i$ such that 
$$\sup_{g'Y \in \mathscr{Y}_i} d_{Haus}(gY, g'Y) = \infty.$$
In either case, there must be an infinite sequence $\{ g_i Y\} \subseteq \mathscr{Y}$ such that $d_{Haus}(g_iY, g_jY)>x_1$ for all $i \neq j$.
However, this violates the conclusion of Lemma \ref{seq_is_fin}.
Thus (*) must hold.
\\

For each $i$, fix a representative $g_i Y \in \mathscr{Y}_i$, and let 
$$\mu > \max_i \sup_{gY \in \mathscr{Y}_i} d_{Haus}(g_iY, gY).$$
Thus, for each $i$ and any $gY, g'Y \in \mathscr{Y}_i$, $d_{Haus}(g_iY, gY)< \mu$ and 

\noindent $d_{Haus}(gY, g'Y)<2\mu$. 

An argument similar to the following was used by Papasoglu in the proof of Lemma 2.3 of \cite{Papa_asym_top}.
Let $\{ C_1, \ldots , C_n\}$ denote the deep components of the complement of $Y$, and let $\{D_1, \ldots , D_m\}$ denote the deep components of the complement of $N_\mu(Y)$.  
If $g$ is a vertex in $N_{k}(Y)$, then $g^{-1}Y$ meets $\overline {B}_{k}(e)$ so $g^{-1}Y \in \mathscr{Y}_i$ for some $i$, and hence $d_{Haus}(g^{-1}Y, g_iY)< \mu$.
Thus note that, for any $j$, there exist $k_1, \ldots , k_l$ such that $g^{-1}C_j$ contains $g_iD_{k_1} \coprod \ldots \coprod g_iD_{k_l}$, and is disjoint from $g_iD_{\hat{k}}$ for all $\hat{k} \notin \{ k_1, \ldots , k_l\}$.
Recalling that all translates of $Y$ satisfy the deep and shallow conditions, it follows that $d_{Haus}(g^{-1}C_j, [g_iD_{k_1} \coprod \ldots \coprod g_iD_{k_l}])<\infty$.

Thus, given a vertex $g \in N_{k}(Y)$, we can define a function 
$$f_g \co \{ C_1, \ldots , C_n\} \to \mathcal{P} \{ D_1, \ldots , D_m\},$$
where $\mathcal{P} \{ D_1, \ldots , D_m\}$ denotes the power set on $\{ D_1, \ldots , D_m\}$, such that $f_g(C_j) = \{ D_{k_1}, \ldots , D_{k_l}\}$ if $g^{-1}Y \in \mathscr{Y}_i$ and $g^{-1}C_j$ is a finite Hausdorff distance from $[g_iD_{k_1} \coprod \ldots \coprod g_iD_{k_l}]$.
Then, if $g,g' \in N_{k}(Y)$, we shall write $g \sim g'$ if $g^{-1}Y, (g')^{-1}Y \in \mathscr{Y}_i$ for some $i$, and $f_g = f_{g'}$.
Note that there are only finitely many equivalence classes.
Note also that if $g \sim g'$, then $2\mu > d_{Haus}(g^{-1}Y, (g')^{-1}Y) = d_{Haus}(g'g^{-1}Y, Y)$, and, 
for each $j$, $d_{Haus}(g^{-1}C_j, (g')^{-1}C_j) = d_{Haus}(g'g^{-1}C_j, C_j)$ is finite.

Let $R>0$ be such that $B_R(e)$ contains a member of each equivalence class from the equivalence relation on the vertices of $N_{k}(Y)$.
Then for each vertex $g \in N_{k}(Y)$, let $\tau_g \in B_R(e)$ denote a vertex of $N_{k}(Y)$ that is equivalent to $g$.
Let $H$ be the subgroup of $G$ that is generated by the elements $g\tau_g^{-1}$ for all vertices $g \in N_{k}(Y)$.
Thus, for each $h \in H$ and for all $j$, both $d_{Haus}(Y, hY)$ and $d_{Haus}(C_j, hC_j)$ are finite.

Note that the vertices of $Y$ are contained in $H(B_R(e)) = N_R(H)$, and hence $Y \subseteq N_{R+1}(H)$.
On the other hand, we claim that $H$ is contained in a uniform neighborhood of $Y$.
To see this, fix $h \in H$ and suppose that $hY$ is contained in some $C_i$.
If also $Y \subseteq hC_{i'}$ for some $i'$ then, for all $j\neq i'$, the region $hC_j$ meets $hY \subseteq C_i$ and does not meet $Y \subseteq hC_{i'}$, hence is contained in $C_i$.
Since $Y$, and hence $hY$, is 3--separating, we have that $C_i$ contains more than one component $hC_j$.  
But $d_{Haus}(C_j, hC_j)<\infty$ for all $j$, so this cannot be the case.

Otherwise, we have that $hY$ meets $Y$, $hY$ meets more than one component $C_i$, or $Y$ meets more than one component $hC_i$.
In the latter two cases, the noncrossing condition implies that there is some vertex $z \in (hY \cap N_{k}(Y))$.
Thus in any case, we have the existence of some such $z$.
Hence $e \in z^{-1}hY$ and $z^{-1}Y$ meets $\overline {B}_{k}(e)$, so $z^{-1}hY, z^{-1}Y \in \mathscr{Y}$.
We noted earlier that $d_{Haus}(Y, hY)<\infty$, so since $d_{Haus}(z^{-1}Y, z^{-1}hY) = d_{Haus}(Y, hY)$, we must have that $z^{-1}Y$ and $z^{-1}hY$ are both in some $\mathscr{Y}_i$ and hence that $d_{Haus}(z^{-1}Y, z^{-1}hY) = d_{Haus}(Y, hY)<2\mu$.
As $h \in hY$, it follows that $h \in N_{2\mu}(Y)$ and thus $H \subseteq N_{2\mu}(Y)$.

Thus $d_{Haus}(H,Y)$ is bounded by $\max \{ R+1, 2\mu\}$.
\end{proof}

We conclude with a couple of observations related to what we saw in Section \ref{ud_section}.

\begin{rmk}\label{Y_ci_H}
In the proof of Theorem \ref{subgroup_thm}, we have given the subgroup $H$ via an infinite generating set.  
However, if $Y$ connected, then it follows from Corollary \ref{Y_conn_implies_H_fg} that $H$ must actually be finitely generated.

In the case that $H$ is finitely generated, we can consider $H$ with respect to its own intrinsic geometry, $(H, d_H)$.
In general we might have that $H$ is badly distorted in $G$, so we cannot expect that $(H,d_H)$ is quasi-isometric to $Y$.
However, it follows from Lemma \ref{inclusion_is_wup} that the two spaces are coarsely isometric.
\end{rmk}

We will see in Section \ref{algebraic_invariants} that certain coarse geometric information about $Y$ implies algebraic information about $H$.

\section{A partial converse to the Flat Torus Theorem}\label{noncrossing_section}

The noncrossing condition is known to be satisfied by ``quasi-lines'' (i.e. uniform neighborhoods of images of $\mathbb{R}$ under uniformly distorting maps) in the setting that we are working in, if they are contained in Cayley graphs of finitely presented, one-ended groups.
Indeed, Proposition 2.1 of \cite{Papa_qlines} shows essentially that any quasi-line contained in the Cayley graph of a finitely presented, one-ended group that satisfies the shallow and 3--separating conditions also satisfies the noncrossing condition.  
(See also \cite{Vav_qlines}.)
(In this setting, the quasi-line will automatically satisfy the deep condition.  See \cite{Vav_qlines}.)

The situation for more general subsets appears to be trickier.
In Proposition \ref{cat0_prop}, we will prove that the noncrossing condition is satisfied in a certain CAT(0) setting.

Recall that the Flat Torus Theorem implies that if a group $G$ acts geometrically on a CAT(0) space $X$, and $H \cong \mathbb{Z}^n$ is a subgroup of $G$, then $X$ contains an isometrically embedded copy of $\mathbb{E}^n$, that $H$ acts on with torus quotient.  

The full converse to the Flat Torus Theorem is false --- that is, if $G$ acts geometrically on a CAT(0) space $X$, and $F_0$ is a Euclidean flat in $X$, then $F_0$ will not necessarily be a finite Hausdorff distance from an orbit of a $\mathbb{Z}^n$ subgroup of $G$, as we saw in Remark \ref{counterexs}.
However, by combining Proposition \ref{cat0_prop} below with Theorem \ref{subgroup_thm}, we will get Theorem \ref{flat_torus_converse}, which is a partial converse to the Flat Torus Theorem.

\begin{prop}\label{cat0_prop}
Let $X$ be a CAT(0) space and let $F_0,F_0'\subseteq X$ be isometrically embedded copies of Euclidean space $\mathbb{E}^n$.
Let $R\geq 0$ and let $F = N_R(F_0)$ and $F' = N_R(F_0')$.
Suppose that both $F$ and $F'$ satisfy $deep(m_0)$ and $shallow(m_1)$, and that $F$ is 3--separating.

Then there is some constant $k' = k'(m_0, m_1, R)$ such that $F$ is contained in the $k'$--neighborhood of a deep component of the complement of $F'$.
\end{prop}

\begin{proof}
We will prove the proposition for $k' = (m_1(0)+m_0(0)+R)$.

We will show that there is some component $C$ of $(X-F')$ such that $F_0 \subseteq (F' \cup C)$.
If $C$ is deep, then $F' \subseteq N_{m_0(0)}(C)$, so $F_0 \subseteq N_{m_0(0)}(C)$, and hence $F \subseteq N_{m_0(0)+R}(C)$.
If $C$ is shallow, then $C \subseteq N_{m_1(0)}(F')$, so $F_0 \subseteq N_{m_1(0)}(F')$.
If $C'$ denotes any deep component of $(X-F')$ then, as $F' \subseteq N_{m_0(0)}(C')$, we have that $F_0 \subseteq N_{m_1(0)+m_0(0)}(C')$ and hence $F \subseteq N_{m_1(0)+m_0(0)+R}(C')$.
Thus in either case it will follow that the proposition holds.

We have that $X$ is CAT(0), hence is uniquely geodesic; for any pair of points $p,q \in X$, we shall denote the geodesic segment connecting $p$ and $q$ by $[p,q]$.  
Note that $F_0$ and $F_0'$ are convex.

Moreover, since $X$ is CAT(0), its distance function is convex, i.e., for any two geodesics $c,c' \co [0,1] \to X$ parameterized proportional to arc length, and for any $t \in [0,1]$,
$$d(c(t), c'(t)) \leq (1-t)d(c(0), c'(0)) + td(c(1), c'(1)).$$
(See Proposition II.2.2 of \cite{BridsonHaefliger}.)
It follows that $F$ and $F'$ are convex, and hence so are $(F' \cap F_0)$ and $(F \cap F_0')$.

Suppose for a contradiction that $F_0$ is not contained in $(F' \cup C)$, for any component $C$ of $(X-F')$.
Then $F_0$ meets two distinct components of $(X-F')$, say $C_1$ and $C_2$.
Thus we have that $(F' \cap F_0)$ separates $F_0$ and is convex in $F_0$.
Note that it follows that $(F' \cap F_0)$ is a uniform neighborhood of a hyperplane in $F_0$.

Let $v_i \in (F_0 \cap C_i)$, for $i=1,2$.
Then there is some $\epsilon > 0$ such that $B_\epsilon (v_i) \subseteq C_i$.
Let $q \in (F' \cap F_0)$, and let $l_i$ denote the geodesic ray in $F_0$ that begins at $q$ and contains $v_i$.
Since $F'$ is convex, note that the subray of $l_i$ that begins at $v_i$ is contained in $C_i$.
Let $r_i=d(v_i, q)$.

Fix $i \in \{ 1,2\}$, and let $w \in l_i$ be such that $d(w,q) = r' > r_i$, let $\epsilon'>0$, and suppose that there is some $p \in (B_{\epsilon'}(w) \cap F')$.
Then the CAT(0) inequality implies that $[p,q]$ meets the $(r_i \epsilon '/r')$--ball about $v_i$.
The convexity of $F'$ implies that $[p,q] \subseteq F'$, hence $\epsilon < (r_i \epsilon'/r')$.
It follows that there is some $w_i \in F_0$ such that the $m_0(R)$--ball about $w_i$ is contained in $C_i$.

On the other hand, we have that $(F \cap F_0')$ is convex in $F_0'$, so in particular $F_0'-(F \cap F_0')$ consists of no more than two components.
Recall that $F$ is 3--separating, thus there is a deep component $X_0$ of $(X-F)$ that does not meet $F_0'$.
Since $F$ satisfies the shallow condition, there is a deep component $X_1$ of $(X-N_{R}(F))$ such that $X_1 \subseteq X_0$.
Since $(X_0 \cap F_0' )= \emptyset$ and $F' = N_R(F_0')$, it follows that $(X_1 \cap F') = \emptyset$.

But the $m_0(R)$--ball about each $w_i$ must meet $X_1$, so $X_1$ is a connected region in $X$ that meets $C_1$ and $C_2$, but not $F'$.
Since $C_1$ and $C_2$ are distinct components of $(X-F')$, this is impossible.
Thus there is some component $C$ of $(X-F')$ such that $F_0 \subseteq (F' \cup C)$.
\end{proof}

Recall the following well-known theorem of Gromov.

\begin{thm}\label{Z^n_thm}
If a finitely generated group is quasi-isometric to $\mathbb{Z}^n$ then it contains $\mathbb{Z}^n$ as a subgroup of finite index.
\end{thm}

Recall also that two subgroups $H_1$ and $H_2$ of a group $G$ are said to be commensurable if $[H_1:H_1 \cap H_2]$ and $[H_2:H_1 \cap H_2]$ are both finite.

\begin{rmk}\label{comm_implies_fin_H_dist}
For any finitely generated group $G$ with commensurable subgroups $H_1$ and $H_2$, it is straightforward to show that the Hausdorff distance between $H_1$ and $H_2$ in $\mathscr{C}(G)$ is finite.
\end{rmk}

Thus we have the following.

\begin{cor}\label{Z^n_thm2}
If $G$ is a finitely generated group and $H$ is a finitely generated subgroup of $G$ that is quasi-isometric to $\mathbb{Z}^n$ (with respect to its intrinsic metric), then $G$ contains a subgroup $H_0 \cong \mathbb{Z}^n$ such that
$$d_{Haus}(H,H_0)<\infty,$$
where we take $d_{Haus}$ to denote Hausdorff distance in $\mathscr{C}(G)$.
\end{cor}

Now we can give our partial converse to the Flat Torus Theorem.

\begin{thm}\label{flat_torus_converse}
Let $X$ be a CAT(0) space, and let $G$ be a finitely generated group acting geometrically on $X$.
Suppose that $X$ contains an isometrically embedded copy, $F_0$, of $\mathbb{E}^n$ that has a uniform neighborhood that satisfies the deep, shallow and 3--separating conditions.

Then $G$ contains a subgroup $H \cong \mathbb{Z}^n$, such that for any $x_0 \in X$, $d_{Haus}(Hx_0, F_0)<\infty$.
\end{thm}

\begin{proof}
Fix $x_0 \in X$, let $\phi \co \mathscr{C}(G) \to X$ be a quasi-isometry that takes each $g \in G$ to $gx_0$, and let $\phi'$ be a quasi-inverse to $\phi$.
Let $F$ be a connected uniform neighborhood of $F_0$ that satisfies the deep, shallow and 3--separating conditions.
By Lemmas \ref{shallow_qi_inv}, \ref{n--sep_qi_inv} and \ref{deep_qi_inv}, there is some connected uniform neighborhood $Y$ of $\phi'(F)$ that satisfies the deep, shallow and 3--separating conditions.
$F$ is quasi-isometric to $\mathbb{R}^n$, so $Y$ is quasi-isometric to $\mathbb{Z}^n$.

Let $r>0$ be such that $\phi(Y) \subseteq N_r(F)$.
As we saw in Lemma \ref{shallow_qi_inv}, we can enlarge $r$ if necessary so that each component of the complement of $N_r(F)$ meets the image under $\phi$ of no more than one component of $(\mathscr{C}(G)-Y)$.
Proposition \ref{cat0_prop} implies that $N_r(F)$ satisfies the noncrossing condition.
It follows that there is some $k_0=k_0(r)>0$ such that for all $g \in G$, $g\phi(Y)$ is contained in the $k_0$--neighborhood of a deep component of $(X-N_r(F))$.

As $d_{Haus}(g\phi(Y), \phi(gY))$ is bounded by a function of the parameters of $\phi$ (and is independent of $g$), there is some constant $k$, depending only on $k_0(r)$ and the parameters of $\phi$, such that for all $g \in G$, $gY$ is contained in the $k$--neighborhood of a deep component of the complement of $Y$.
Hence $Y$ satisfies $noncrossing(k)$.
It follows from Theorem \ref{subgroup_thm} and Corollary \ref{Y_conn_implies_H_fg} that $G$ contains some finitely generated subgroup $H$ such that $d_{Haus}(Y,H)<\infty$, and hence $d_{Haus}(F,Hx_0)<\infty$.

Note that by Proposition \ref{ci_is_qi}, $H$ is quasi-isometric to $F_0$, and by Corollary \ref{Z^n_thm2}, we can assume that $H \cong \mathbb{Z}^n$.
\end{proof}

We end this section by mentioning that we expect the noncrossing condition to hold in far more general settings.  
Specifically, we expect the following.

\begin{conj}\label{noncrossing_conj}
Let $G$ be a finitely generated group, and let $Y$ be a 3--separating connected subset of $\mathscr{C}(G)$.
If for all subsets $Y' \subseteq Y$ with $d_{Haus}(Y', Y)=\infty$, $Y'$ is not 2--separating, and $G$ is of type $F_{n}$ for sufficiently large $n \in (\mathbb{N} \cup \{ \infty\})$ depending on the geometry of $Y$, then $Y$ satisfies the noncrossing condition.
\end{conj}

The reader should note that results in this direction could be very interesting when combined with the results from this paper.
When combined with Theorem \ref{subgroup_thm}, such results could yield a subgroup detection theorem that depends only on coarse geometry.
If combined with Theorem \ref{main_thm}, such a result could give the quasi-isometry invariance of certain types of subgroups.

\section{Some coarse isometry invariants}\label{algebraic_invariants}

We showed in Theorem \ref{subgroup_thm} that certain properties of a subset $Y$ of a Cayley graph $\mathscr{C}(G)$ imply that there is some subgroup $H$ of $G$ such that $d_{Haus}(Y,H)<\infty$.
In Remark \ref{Y_ci_H}, we noted that $H$ is finitely generated if $Y$ is connected, and in this case $Y$ is coarsely isometric to $(H, d_H)$.
In this section, we will consider a couple of basic invariants of coarse isometries, in order to see that the coarse geometry of $Y$ determines aspects of the algebraic structure of $H$.
The first such invariant we will consider is coarse $n$--connectedness.

If $(X,d)$ is a discrete metric space and $\epsilon \geq 0$, then we use $Rips_\epsilon (X)$ to denote the $\epsilon$--Rips complex of $X$.
Thus $Rips_\epsilon (X)$ is the simplicial complex with vertex set equal to $X$, and such that any finite subset $X_0$ of $X$ spans a simplex if and only if, for all $x_1, x_2 \in X_0$, $d(x_1, x_2) \leq \epsilon$.

The following is Definition 2.10 of \cite{Kap_ggt_book}.
\begin{defn}
A discrete metric space $X$ is said to be {\em coarsely $n$--connected} if, for each $r \geq 0$, there is some $R \geq r$ such that the natural simplicial map $Rips_r(X) \to Rips_R(X)$ induces the trivial map on $i^{th}$ homotopy groups, for every $0 \leq i \leq n$.
\end{defn}

Note that in the case that $X$ is a discrete subset of a Cayley graph, or more generally of a geodesic space, this definition of coarse 0--connectedness agrees with that given in Definition \ref{0-conn_defn}.

The next theorem appears in \cite{Vav_qlines}, and is based on Corollary 2.15 of \cite{Kap_ggt_book}.

\begin{thm}\label{fn_ci_inv}{\rm {\bf \cite{Vav_qlines}}} 
Coarse $n$--connectedness is a coarse isometry invariant.
\end{thm}

The proof of Theorem 2.21 of \cite{Kap_ggt_book} shows that each coarsely $n$--connected group is of type $F_{n+1}$, which gives the next result.

\begin{cor}
Let $G,Y$ and $H$ be as in Theorem \ref{subgroup_thm}, assume that $H$ is finitely generated, and let $n>0$.
Then $Y$ is coarsely $n$--connected if and only if $H$ is of type $F_{n+1}$.
\end{cor}

Next, we will use Gromov's theorem about groups with polynomial growth to see that, assuming $H$ is finitely generated, $H$ has a nilpotent subgroup of finite index if $Y$ ``coarsely'' has slow growth in $\mathscr{C}(G)$.
For the remainder of this section, if $(X,d)$ is a metric space, $x \in X$ and $n > 0$, then let $B_n(x,(X,d))$ denote the $n$--ball about $x$ in $(X,d)$, and let $\overline{B}_{n}(x,(X,d))$ denote the closure of that ball.

Recall that, if $(X,d)$ is a discrete metric space and $x \in X$, then the growth function of $X$ with respect to the basepoint $x$ is defined by $\beta(x, (X,d))(n) := \# \overline{B}_n(x,(X,d))$.
If $H$ is a group with generating set $S$, giving rise to the metric we denote by $d_H$, then $(H, d_H)$ has a growth function that is independent of basepoints, and which we denote by $\beta_H(n)$.

Following \cite{delaHarpe_book}, we say that a function $\beta \co \mathbb{R}_+ \to \mathbb{R}_+$ is weakly dominated by another function $\beta' \co \mathbb{R}_+ \to \mathbb{R}_+$, denoted $\beta \stackrel{w}{\prec} \beta'$, if there are constants $\Lambda \geq 1$ and $C \geq 0$ such that 
$$\beta(n) \leq \Lambda \beta'(\Lambda n + C) + C$$
for all $n>0$.
Note that $\stackrel{w}{\prec}$ is a transitive relation.
We will say that a discrete metric space $(X,d)$ has polynomial growth if there is some $a \geq 0$ such that $\beta(x, (X,d))(n)$ is weakly dominated by the function $n \mapsto n^a$.

\begin{lem}\label{weakly_dominated}
Let $H$ be a finitely generated subgroup of a finitely generated group $G$, and let $Y \subseteq \mathscr{C}(G)$ be such that $d_{Haus}(Y,H)<\infty$.
If $\pi \co \mathscr{C}(G) \to G$ denotes a nearest point projection map, then $\beta_H \stackrel{w}{\prec} \beta(x,(\pi(Y), d_G))$ for any $x \in \pi(Y)$.
\end{lem}

\begin{proof}
Let $L>0$ be such that the identity map $i_H \co (H, d_H) \to (H, d_G)$ is $(\phi, \Phi)$--uniformly distorting, with $\Phi(r) \leq Lr$ for all $r$, as in Lemma \ref{inclusion_is_wup}.
Let $\rho \co H \to Y$ be a nearest point projection map, so $\rho$ and the projection map $\pi$ move any given point a distance of no more than $d_{Haus}(Y,H)$ and $\frac{1}{2}$ respectively, and let $\zeta=\pi \rho i_H $, so $\zeta \co (H,d_H) \to \pi(Y) \subseteq G$.

For any $h,h' \in H$, note that 
$$d_H(h,h') \leq r \Rightarrow d_G(h, h') \leq Lr $$
$$\Rightarrow d(\zeta(h), \zeta(h')) \leq Lr + 2d_{Haus}(Y,H) + 1.$$
It follows that, for any $h \in H$, 
$$\zeta(\overline{B}_{r}(h,(H, d_H))) \subseteq \overline{B}_{ Lr + 2d_{Haus}(Y,H) + 1}(\zeta(h), (\pi(Y),d_G)).$$

On the other hand, note that for any $h,h' \in H$, 
$$d_G(\zeta(h), \zeta(h')) = 0 \Rightarrow d_G(h,h') \leq  2d_{Haus}(Y,H) + 1,$$
and hence $d_H(h,h')<r_0$ for any $r_0$ such that $\phi (r_0) >  2d_{Haus}(Y,H) + 1$.
Thus $\zeta $ maps no more than $\beta_{H}(r_0)$ elements of $H$ to any given point of $\pi (Y)$.
It follows that, for any $r \geq 0$, and any point $x \in \pi(Y)$,
$$\beta_{H}(r) \leq \beta_{H}(r_0)\beta( x,(\pi(Y),d_G))(Lr + 2d_{Haus}(Y,H) + 1).$$
Hence $\beta_H \stackrel{w}{\prec} \beta(x,(\pi(Y), d_G))$ for any $x \in \pi(Y)$.
\end{proof}

In particular, it follows that if $(\pi(Y),d_G)$ has polynomial growth, then so does $H$.
We recall Gromov's famous theorem about groups with polynomial growth:

\begin{thm}\label{gromov's_thm}
{\em {\bf \cite{Gromov_poly_growth}}}
Let $H$ be a finitely generated group.
Then $(H,d_H)$ has polynomial growth if and only if $H$ has a nilpotent subgroup of finite index.
\end{thm}

Thus we have the following.

\begin{cor}\label{v_nilp}
Let $G,Y$ and $H$ be as in Theorem \ref{subgroup_thm}, assume that $H$ is finitely generated, and let $\pi \co \mathscr{C}(G) \to G$ be a nearest point projection map.
If $(\pi(Y), d_G)$ has polynomial growth, then $H$ has a nilpotent subgroup of finite index.
\end{cor}

\section{Quasi-isometry invariance}\label{final_section}

Using Theorem \ref{subgroup_thm} and a few results below, in this section we will give sufficient conditions for certain subgroups to be invariant under quasi-isometries.
We will begin by showing that certain algebraic hypotheses on a group $G$ and a subgroup $H$ imply that the deep, shallow and 3--separating conditions are satisfied by a uniform neighborhood of $H$ in $\mathscr{C}(G)$.
We saw in Section \ref{terminology_section} that these conditions are essentially preserved under quasi-isometries, so if $f \co \mathscr{C}(G) \to \mathscr{C}(G')$ is a quasi-isometry, $H$ is a subgroup of $G$ that satisfies these algebraic hypotheses, and assuming that the noncrossing condition is suitably satisfied, the existence of a subgroup $H'$ of $G'$ such that $d_{Haus}(H', f(H))<\infty$ will follow from Theorem \ref{subgroup_thm}.

We will conclude the section by seeing that work from \cite{Vav_qlines} implies that the commensurizers of such subgroups are also invariant under quasi-isometries, given an assumption involving the noncrossing condition.

First, we show that any subgroup in a Cayley graph satisfies the shallow condition.

\begin{lem}\label{shallow_for_subgroups}
Let $G$ be a finitely generated group, let $H$ be a subgroup of $G$ and let $R \geq 0$.
Then there is some constant $m_1$, depending on $R$, such that all shallow components of $\mathscr{C}(G)- N_R(H)$ are contained in the $m_1$--neighborhood of $N_R(H)$.

In particular, it follows that $H$ satisfies the shallow condition.
\end{lem}

\begin{proof}
Let $S$ be a shallow component of $(\mathscr{C}^ 1(G)-N_R(H))$, so $S$ projects to a finite component of $(H\backslash \mathscr{C}(G)-H\backslash N_R(H))$.
Recall that $H\backslash \mathscr{C}(G)$ is locally finite, so, as $H \backslash N_R(H)$ is finite, there are only finitely many components of $(H\backslash \mathscr{C}(G)-H\backslash N_R(H))$.
In particular there is some $m_1 \geq 0$ such that all finite components of $(H\backslash \mathscr{C}(G)-H\backslash N_R(H))$ are contained in the $m_1$--neighborhood of $H\backslash N_R(H)$.
If follows that $S$ is contained in the $m_1$--neighborhood of $N_R(H)$.
\end{proof}

Next, we will see that the $n$-separating condition is detectable from an algebraic property of $H<G$.
To give a careful statement, we must make some definitions.
For two subsets $X,Y \subseteq G$, let $X+Y$ denote the symmetric difference of $X$ and $Y$.

\begin{defn}
Following \cite{ScottSwarup_intersection_numbers}, if $G$ is a finitely generated group, $H$ a subgroup of $G$ and $X$ a subset of $G$, then we say that $X$ is {\em $H$--finite} if $X$ is contained in finitely many cosets $Hg$ of $H$, or equivalently if $X \subseteq N_r(H)$ for some $r \geq 0$.
If $X$ is not $H$--finite, then we say that $X$ is {\em $H$--infinite}.
\end{defn}

Let $\mathcal{P}(G)$ denote the power set of all subsets of $G$, let $\mathcal{F}_H(G)$ denote the set of all $H$--finite subsets of $G$, and consider the quotient set $\mathcal{P}(G)/\mathcal{F}_H(G)$, where $X,Y \in \mathcal{P}(G)$ are considered equivalent if $X+Y$ is $H$--finite.
This set forms a vector space over $\mathbb{F}_2$, the field with two elements, under the operation of symmetric difference.
In addition, the set admits an action of $G$ on the right.
The fixed set under this action, $(\mathcal{P}(G)/\mathcal{F}_H(G))^G$, consists of equivalence classes with representatives $X$ such that $X+Xg$ is $H$--finite for all $g \in G$, and it forms a subspace of $\mathcal{P}(G)/\mathcal{F}_H(G)$.
The following definition is due to Kropholler and Roller \cite{KrophollerRoller}.

\begin{defn}
Let $G$ be a finitely generated group and let $H$ be a subset of $G$.  Then define 
$$\tilde{e}(G,H) = \dim_{\mathbb{F}_2}(\mathcal{P}(G)/\mathcal{F}_H(G))^G.$$
\end{defn}

Following Bowditch \cite{BowJSJ}, we shall call $\tilde{e}(G,H)$ the number of {\em coends} of $H$ in $G$.  
(Kropholler and Roller called $\tilde{e}(G,H)$ the number of relative ends of $H$ in $G$.)

Note that $\tilde{e}(G,H)=0$ if and only if $G$ is $H$--finite, i.e. $[G:H]<\infty$.

\begin{rmk}
It is a useful and well-known fact that a subset $X$ of $G$ represents an element of $(\mathcal{P}(G)/\mathcal{F}_H(G))^G$ if and only if $H \backslash \delta X$ is a finite set of edges in the quotient graph $H \backslash \mathscr{C}(G)$, where we use $\delta$ to denote the coboundary operator in $\mathscr{C}(G)$.
Thus $X$ represents an element of $(\mathcal{P}(G)/\mathcal{F}_H(G))^G$ if and only if there is some $r \geq 0$ such that $\delta X$ is contained in the $r$--neighborhood of $H$ in $\mathscr{C}(G)$.
\end{rmk}

The next lemma shows that coends have a natural geometric interpretation, which is closely related to the $n$--separating condition.
In light of Lemma \ref{shallow_for_subgroups}, \cite{Vav_qlines} provides an argument for it.

\begin{lem}\label{coends_and_separating}
Let $G$ be a finitely generated group, let $H$ be a subgroup of $G$ and let $n>0$.
Then $\tilde{e}(G,H) \geq n$ if and only if there is some $R>0$ such that $N_R(H)$ is $n$--separating in $\mathscr{C}(G)$.

Moreover, $\tilde{e}(G,H) = \infty$ if and only if, for each $n>0$, there is some $R=R(n)$ such that $N_R(H)$ is $n$--separating.
\end{lem}

Next we show that uniform neighborhoods of $H$ satisfy the deep condition, as long as ``smaller'' subgroups coarsely do not separate $G$.

\begin{lem}\label{deep_for_subgroups}
Let $G$ be a finitely generated group with subgroup $H$, and suppose that, for all subgroups $K$ of infinite index in $H$, $\tilde{e}(G, K) = 1$.
Then $H$ satisfies the deep condition.
\end{lem}

Our hypothesis about subgroups of $H$ is generally a necessary one.    
In the special case that $H \cong \mathbb{Z}$, we are imposing the condition that $e(G)=1$;  for a counterexample in the absence of this assumption, consider the free group $F_n$, $n>1$, with respect to a standard generating set, and let $H$ be the subgroup generated by one of the standard generators.  
Then certainly $H$ will not satisfy the deep condition.

Generalizing this counterexample, let $G$ be an amalgamated free product $A *_C B$, where $A \cong B \cong \mathbb{Z}^n$ and $C \cong \mathbb{Z}^{n-1}$.  Let $H = A$, and then $H$ does not satisfy the deep condition, with the problem stemming from the fact that $\mathscr{C}(G)$ is coarsely separated by ``smaller'' regions, in particular the subgroup $C$.

\begin{proof}[Proof of Lemma \ref{deep_for_subgroups}]
Fix $r\geq 0$, and it shall suffice to show that there is some constant $m_0(r)$ such that, for any $p \in N_r(H)$, $B_{m_0(r)}(p)$ meets all the deep components of the complement of $N_r(H)$.

If $(\mathscr{C}(G)-N_r(H))$ has no deep components, i.e. $H$ is of finite index in $G$, then the deep condition is vacuously satisfied.
Suppose then that there is only one deep complementary component of $N_r(H)$, say $D$.
Then $hD = D$ for all $h \in H$, since in general $H$ acts by permuting the deep complementary components of $N_r(H)$.  
Let $M>0 $ be such that $d(e, D)<M$, and it follows that, for any $h \in H$, $d(h,D)<M$ and hence for any $p \in N_r(H)$, $d(p, D)<(M+r)$.
Hence in this case, our claim follows with $m_0(r) = (M+r)$.

So suppose that $N_r(H)$ is 2--separating.
We will use the assumption on subgroups of $H$ to first show that the frontier of any deep component is a finite Hausdorff distance from $H$.
Then we will show that there are only finitely many possibilities for such distances that can be attained, and from this the lemma will follow.

Let $D$ be a deep component of $(\mathscr{C}(G)-N_r(H))$.  
We claim that the subgroup of $H$ that stabilizes $D$, $stab_H(D)$, is of finite Hausdorff distance from $fr(D)$.
Note that $stab_H(D)$ is contained in a uniform neighborhood of $fr(D)$, since, for any $h \in stab_H(D)$, $d(e,fr(D)) = d(h, fr(D))$.
Thus we must show that $fr(D)$ is contained in a uniform neighborhood of $stab_H(D)$.
It suffices to assume that $fr(D)$ is infinite.  

First we note that each point of $fr(D)$ is of distance $r$ from some point of $H$.
Let $fr(D) = \{ d_i\}$ and, for each $i$, let $h_i$ be a point of $H$ such that $d(d_i, h_i) = r$.
Then, for each $i$, $h_i^{-1}D$ is a deep component of $(\mathscr{C}(G)-N_r(H))$ that meets the closed ball $\overline{B}_r(e)$. 
It follows that, for all $i,j$, either $h_i^{-1}D = h_j^{-1}D$, or the two regions are disjoint.
As $\overline{B}_r(e)$ is finite, $\{ h_i^{-1}D\}$ must be a finite collection of deep components of the complement of $N_r(H)$, say equal to $\{D_1, \ldots , D_n\}$.

For each $j$, choose $k_j \in H$ such that $D_j = k_jD$, and hence for all $i, j$ such that $h_i^{-1}D = k_jD = D_j$, $h_ik_jD = D$, i.e. $h_ik_j \in stab_H(D)$.
As $d(h_ik_j, h_i) = |k_j|$, if $t = \max_j |k_j|$ then it follows that $fr(D)$ is contained in the $(r+t)$--neighborhood of $stab_H(D)$.
Hence $fr(D)$ is a finite Hausdorff distance from $stab_H(D)$.

Let $N$ be a uniform neighborhood of $stab_H(D)$ that contains $fr(D)$.  
Then by Lemma \ref{shallow_for_subgroups}, $N$ satisfies the shallow condition, hence $(N\cup D)$ contains a deep component of the complement of $N$.
Since $N_r(H)$ is 2--separating, it follows that the complement of $N$ contains another deep component, so, by Lemma \ref{coends_and_separating}, $\tilde{e}(G,stab_H(D))>1$.
But this contradicts our hypothesis about subgroups of $H$, unless $stab_H(D)$ is of finite index in $H$.
Hence this index is finite, so $stab_H(D)$ is a finite Hausdorff distance from $H$, and thus $fr(D)$ is a finite Hausdorff distance from $H$ as well.

Next, we claim that there is a bound on the Hausdorff distances of the frontiers of these deep components to $H$.  
For suppose instead that $(\mathscr{C}(G)-N_r(H))$ has deep components $D_1, D_2, \ldots$ such that $d_{Haus}(H, fr(D_i)) \to \infty$.
By passing to a subsequence, we shall assume that the values of $d_{Haus}(H, fr(D_i))$ are all distinct.

Note that, for any $i$ and any $h \in H$, $d_{Haus}(H, fr(hD_i)) = d_{Haus}(hH, h\cdot fr(D_i)) = d_{Haus}(H, fr(D_i))$ and hence, for any $h,h' \in H$ and $i \neq j$, $hD_i \neq h'D_j$.  
Note also that $hD_i$ and $h'D_j$ are deep components of $(\mathscr{C}(G)-N_r(H))$, and thus must be disjoint.

Recall that $fr(D_i) \subseteq N_r(H)$, and let $h_i \in H$ be such that $fr(D_i)$ meets $\overline{B}_r(h_i)$ for each $i$, and consider $\{ h_i^{-1}D_i\}$.
This set is a collection of disjoint connected regions of $\mathscr{C}(G)$, all of which meet $\overline{B}_{r}(e)$.
As $\mathscr{C}(G)$ is locally finite, we have reached a contradiction.

Thus there must be a uniform bound on $d_{Haus}(H,D)$ for all deep components $D$ of $(\mathscr{C}(G)-N_r(H))$.
Let $M$ denote this bound.
Then any point in $N_r(H)$ is of a distance no more than $(M+r)$ from any deep component $D$, so our claim follows for $m_0(r) = (M+r)$.
Thus $H$ satisfies the deep condition.
\end{proof}

We will also need the next lemma.

\begin{lem}\label{subgroups_are_ci}
Let $f\co \mathscr{C}(G) \to \mathscr{C}(G')$ be a quasi-isometry between Cayley graphs of finitely generated groups, with $H$ a finitely generated subgroup of $G$ and $H'$ a finitely generated subgroup of $G'$ such that
$$d_{Haus}(H', f(H))< \infty.$$
Then $(H,d_H)$ and $(H', d_{H'})$ are coarsely isometric.
\end{lem}

\begin{proof}
Let $\pi \co f(H) \to H'$ denote nearest point projection, and let 
$$\iota_H \co (H, d_H) \to (H, d_G), \ \iota_{H'} \co (H', d_{H'}) \to (H', d_{G'})$$
denote the identity maps.
By Lemma \ref{inclusion_is_wup}, $\iota_H$ and $\iota_{H'}$ are coarse isometries; let $\iota_{H'}'$ be a coarse inverse to $\iota_{H'}$.  
Note that $\pi$ is a coarse isometry, and $f|_H$ is a quasi-isometry onto its image.

It follows that $(\iota_{H'}')\pi f \iota_H \co (H,d_H) \to (H', d_{H'})$ is a coarse isometry.
\end{proof}

Recall the characterization of subgroups being finitely generated, given in terms of coarse 0--connectedness by Proposition \ref{subgroup_fg}.
It follows from Theorem \ref{fn_ci_inv} and Lemma \ref{subgroups_are_ci} that, if $f \co \mathscr{C}(G) \to \mathscr{C}(G')$ is a quasi-isometry, and $H$ and $H'$ are subgroups of $G$ and $G'$ respectively such that $d_{Haus}(H', f(H))<\infty$, then $H$ is finitely generated if and only if $H'$ is finitely generated.
Immediate from this observation and Proposition \ref{ci_is_qi} is the following:

\begin{prop}\label{ci_subgps_are_qi}
If $f \co \mathscr{C}(G) \to \mathscr{C}(G')$ is a quasi-isometry between Cayley graphs of finitely generated groups, and if $H$ and $H'$ are subgroups of $G$ and $G'$ respectively such that
$$d_{Haus}(H',f(H))<\infty,$$
then $H$ is finitely generated if and only if $H'$ is finitely generated, and in this case $H$ and $H'$ are quasi-isometric.
\end{prop}

We can now prove the main result of this section:

\begin{thm}\label{main_thm}
Let $G$ and $G'$ be finitely generated groups and let $f\co \mathscr{C}(G) \to \mathscr{C}(G')$ be a quasi-isometry.
Suppose that $H$ is a subgroup of $G$ such that $\tilde{e}(G,H) \geq 3$, and, for any infinite index subgroup $K$ of $H$, $\tilde{e}(G,K)=1$.

If sufficiently large uniform neighborhoods of $f(H)$ in $\mathscr{C}(G')$ satisfy the noncrossing condition, then $G'$ contains a subgroup $H'$ such that 
$$d_{Haus}(H', f(H))<\infty.$$
In addition, $H$ is finitely generated if and only if $H'$ is finitely generated, and in this case $H$ is quasi-isometric to $H'$.
\end{thm}

\begin{proof}
Let $G, G', H$ and $f$ be as stated.
By Lemmas \ref{shallow_for_subgroups} and \ref{deep_for_subgroups}, $H$ satisfies the shallow and deep conditions, hence so does any uniform neighborhood of $H$ in $\mathscr{C}(G)$.
We have $\tilde{e}(G,H) \geq 3$, so by Lemma \ref{coends_and_separating}, there is some $R>0$ such that $N_R(H)$ is 3--separating.
It follows from Lemmas \ref{shallow_qi_inv}, \ref{n--sep_qi_inv} and \ref{deep_qi_inv} that there is some uniform neighborhood $Y$ of $f(H)$ that satisfies the deep, shallow and 3--separating conditions.

We have assumed that sufficiently large uniform neighborhoods of $f(H)$ satisfy the noncrossing condition, hence, by replacing $Y$ with a bigger uniform neighborhood of $f(H)$ if necessary, we have that $Y$ satisfies the noncrossing condition.
Note that replacing $Y$ in this manner will not change that $Y$ satisfies the deep, shallow, and 3--separating conditions.
Thus, by Theorem \ref{subgroup_thm}, $G'$ contains a subgroup $H'$ such that $d_{Haus}(Y,H')<\infty$, hence $d_{Haus}(H', f(H))<\infty$.

Finally, by Proposition \ref{ci_subgps_are_qi}, $H$ is finitely generated if and only if $H'$ is also finitely generated, and in this case the two subgroups are quasi-isometric.
\end{proof}

We now turn our attention to commensurizer subgroups.

\begin{defn}
If $H_1$ and $H_2$ are subgroups of a group $G$, then they are said to be {\em commensurable} if $[H_1:H_1 \cap H_2]<\infty$ and $[H_2:H_1 \cap H_2] <\infty$.
If $H$ is a subgroup of $G$, then the {\em commensurizer} of $H$ in $G$ is defined to be the subgroup $Comm_G(H) = \{ g \in G : H$ and $g^{-1}Hg$ are commensurable$\}$.
\end{defn}
The commensurizer has the following geometric characterization.

\begin{lem}\label{comm_char}
{\rm {\bf \cite{Vav_qlines}}}
If $G$ is a finitely generated group with subgroup $H$, then 
$$Comm_G(H) = \{ g \in G : d_{Haus}(H, gH)< \infty\}.$$
\end{lem}

The proof of next proposition follows that of an analogous result in \cite{Vav_qlines}.

\begin{prop}\label{comm_prop}
Let $G$ be a finitely generated group and let $\mathscr{Y}$ be a collection of pairwise finite Hausdorff distance 3--separating subsets of $\mathscr{C}(G)$ that satisfy $deep(m_0)$ and $shallow(m_1)$.
Suppose that $\{ gY \}_{g \in G, Y \in \mathscr{Y}}$ satisfies the noncrossing condition.
Finally, fix any $Y \in \mathscr{Y}$ and by Theorem \ref{subgroup_thm} there is a subgroup $H$ of $G$ such that $d_{Haus}(Y,H)<\infty$.

Then there is some $R_1\geq 0$ such that $\mathscr{Y} \subseteq N_{R_1}(Comm_G(H))$.
\end{prop}

This result, together with Theorem \ref{main_thm} and a further assumption about the noncrossing condition being satisfied, imply the quasi-isometry invariance of the commensurizers of the subgroups under discussion:

\begin{cor}\label{comm_cor}
Let $G,G', H, H'$ and $f$ be as in Theorem \ref{main_thm}.
Let $f'$ be a quasi-inverse to $f$, and suppose there is some $R_0 >0$ such that, for all $R> R_0$, both $$\{ gN_R(f'(c'H'))\}_{g \in G, c' \in Comm_{G'}(H')}$$ and $$\{ g' N_R(f(cH))\}_{g' \in G', c \in Comm_G(H)}$$ satisfy the noncrossing condition.

Then $$d_{Haus}(Comm_{G'}(H'), f(Comm_G(H)))<\infty.$$
\end{cor}

\begin{proof}
As we saw in the proof of Theorem \ref{main_thm}, there is some $R>R_0$ such that $N_R(f(H))$ satisfies the deep, shallow and 3--separating conditions.
We can further assume that $R$, $m_0$ and $m_1$ are such that for all $g\in G$, $N_R(f(gH))$ is 3--separating and satisfies $deep(m_0)$ and $shallow(m_1)$.

For all $c \in Comm_{G}(H)$, $d_{Haus}(cH, H)<\infty$ and hence the elements of $\{ N_R(f(cH))\}_{c \in Comm_G(H)}$ are of pairwise finite Hausdorff distance.
By Proposition \ref{comm_prop}, there must be some $R_1>0$ such that 
$$\bigcup_{c \in Comm_G(H)}  N_R(f(cH))=N_R(f(Comm_G(H))) \subseteq N_{R_1}(Comm_{G'}(H')).$$

Similarly there are $R', R_1'>0$ such that 
$$\bigcup_{c' \in Comm_{G'}(H')} N_{R'}(f'(c'H'))=N_{R'}(f'(Comm_{G'}(H'))) \subseteq N_{R_1'}(Comm_G(H)).$$
Hence $d_{Haus}(Comm_{G'}(H'), f(Comm_G(H)))<\infty.$
\end{proof}

Corollary \ref{comm_cor} generalizes the main result of \cite{Vav_qlines}, which proves the corollary in the case that $H \cong \mathbb{Z}$, and $G,G'$ are finitely presented. 
(The noncrossing condition is always satisfied in that setting, by Proposition 2.1 of \cite{Papa_qlines}.)

We note that the commensurizer subgroups in Corollary \ref{comm_cor} will in general not be finitely generated.
The corollary, together with Proposition \ref{ci_subgps_are_qi}, implies the following.

\begin{cor}\label{comm_fg}
Let $H$ and $H'$ be as in Corollary \ref{comm_cor}.
Then $Comm_G(H)$ is finitely generated if and only if $Comm_{G'}(H')$ is finitely generated.
\end{cor}

\section{Splittings}\label{splittings_section}

The settings of Theorems \ref{subgroup_thm} and \ref{main_thm} imply the existence of splittings of the ambient groups, as we will see in this section.
It will follow that, assuming the satisfaction of the noncrossing condition as before, a large class of splittings are invariant under quasi-isometries.

For the arguments in this section, we will need to introduce the number of ends of a group, the number of ends of a pair, almost invariant sets, and related notions.
For this we follow Scott and Swarup (see, for instance, \cite{ScottSwarup_intersection_numbers}).

For any locally finite, connected simplicial complex $X$, define the number of ends of $X$ to be 
$$e(X) = \sup \# \{ {\rm infinite \ components \  of \  } (X-K)\},$$
where the supremum is taken over all finite subcomplexes $K$ of $X$.
Thus $e(X)$ may take any value in $\mathbb{Z}_{\geq 0} \cup \{ \infty\}$.

For any finitely generated group $G$, the number of ends of $G$, $e(G)$, is defined to be $e(\mathscr{C}(G))$.
The number of ends is invariant under quasi-isometries, and therefore this definition does not depend on our choice of finite generating sets for $G$.

It is a fact due to Hopf \cite{Hopf_ends} that the number of ends of a finitely generated group can be only $0,1,2$ or $\infty$.
We have that $e(G) = 0$ if and only if $G$ is finite, $e(G)=2$ if and only if $G$ has a finite index $\mathbb{Z}$ subgroup, and by Stallings' Theorem, $G$ splits as an amalgamated free product or HNN extension over a finite subgroup if and only if $e(G) = 2$ or $\infty$.
See \cite{ScottWall} for more details.

\begin{defn}
If $H$ is a subgroup of a finitely generated group $G$, then {\em the number of ends of the pair $(G,H)$}, denoted $e(G,H)$, is defined to be the number of ends of the quotient graph $H \backslash \mathscr{C}(G)$.
\end{defn}

We defined $\tilde{e}(G,H)$ in the previous section, and showed that, though it is defined algebraically, its value can be detected from coarse separation properties of neighborhoods of $H$ in $\mathscr{C}(G)$ (Lemma \ref{coends_and_separating}).
In fact $\tilde{e}(G,H)$ is closely related to $e(G,H)$.
For instance, in \cite{KrophollerRoller}, Kropholler and Roller showed that $e(G,H) \leq \tilde{e}(G,H)$, and that, if $H \lhd G$, then $\tilde{e}(G,H) = e(G,H) = e(G/H)$.

Recall that a subset of $G$ is said to be $H$--finite if it is contained in finitely many cosets $Hg$, and that this condition is equivalent to the subset being contained in some uniform neighborhood of $H$.  
We say that a subset is $H$--infinite if it is not $H$--finite.

We say that subsets $X$ and $Y$ of $G$ are $H$--almost equal if their symmetric difference, $X+Y$, is $H$--finite.
Let $X^*$ denote $(G-X)$.
We will say that $X$ is nontrivial if neither $X$ nor $X^*$ is $H$--finite.
We will say that $X,Y \subseteq G$ are $H$--almost complementary if $X^*$ and $Y$ are $H$--almost equal, and we will say that $X$ is $H$--almost contained in $Y$, denoted $X \stackrel{H}{\subseteq}Y$, if $X \cap Y^*$ is $H$--finite.

\begin{defn}
If $G$ is a group, $H$ is a subgroup of $G$ and $X$ is a subset of $G$, then $X$ is said to be {\em $H$--almost invariant} if $X$ is invariant under the left action of $H$ and represents an element of $(\mathcal{P}(G)/\mathcal{F}_H(G))^G$.
That is, $X$ is $H$--almost invariant if $X=HX$ and $X+Xg$ is $H$--finite for all $g \in G$.
\end{defn}

($H$--almost invariant sets are related to $e(G,H)$ --- in fact we could have defined $e(G,H)$ as we did $\tilde{e}(G,H)$, replacing $(\mathcal{P}(G)/\mathcal{F}_H(G))^G$ with equivalence classes of $H$--almost invariant sets.)

We now turn our attention to group splittings.
Recall that a group $G$ is said to split over a subgroup $H$ if $G$ can be written as an amalgamated free product $A*_HB$, with $A \neq H \neq B$, or as an HNN extension $A*_H$.
Equivalently, $G$ splits over $H$ whenever $G$ acts on a simplicial tree $T$ without edge inversions or any proper $G$--invariant subtrees, and with $H$ the stabilizer of some edge of $T$.
(See, for instance, \cite{Serre_trees_orig} or \cite{ScottWall}.)

The numbers of ends of pairs are related to splittings by \cite[Lemma 8.3]{ScottWall}, which shows that if $G$ is a finitely generated group that splits over a subgroup $H$ then $e(G,H) \geq 2$.
(The converse to this statement is false in general.  Understanding when the existence of a subgroup $H$ of $G$ such that $e(G,H) \geq 2$ implies the existence of a splitting of $G$ is an important question, on which much work has been done.
See \cite{Wall_survey} for a survey.)

As $e(G,H) \leq \tilde{e}(G,H)$, we have the following.

\begin{prop}\label{splitting_coends}
If $G$ is a finitely generated group that splits over $H$, then $\tilde{e}(G,H) \geq 2$.
\end{prop}

\begin{defn}
Suppose that $G$ admits a splitting over a subgroup $H$.  
We say that this splitting {\em has three coends} if $\tilde{e}(G,H) \geq 3$.
\end{defn}

The subgroup $H$ associated to a splitting of $G$ is well-defined up to conjugacy.
So to see that this definition makes sense, note that any inner automorphism of $G$ is a quasi-isometry, so Lemmas \ref{n--sep_qi_inv} and \ref{coends_and_separating} imply that $\tilde{e}(G,H) = \tilde{e}(G,gHg^{-1})$ for any $g \in G$.

Our work in the previous sections will imply the existence and quasi-isometry invariance of splittings with three coends, under the appropriate hypotheses.
Thus it is relevant to investigate these types of splittings --- we provide a characterization in Theorem \ref{etilde_equals_2} below.
Before stating this result, we must define the notion of interlaced cosets.


\begin{defn}
Let $H$ be an infinite index subgroup of a finitely generated group $G$.  Then we say that $H$ {\em has interlaced cosets in $G$} if, for every $r>0$ and every pair of deep components $D, D'$ of $\mathscr{C}(G)-N_r(H)$ there is a sequence $g_1, \ldots , g_n \in G$ such that $g_iH$ meets components $C_i, C_i'$ of the complement of $N_r(H)$, $C_1 = D$, $C_n' = D'$ and $C_i' = C_{i+1}$ whenever $1 \leq i < n$.
\end{defn}

Note that in particular if $\tilde{e}(G,H)= 1$, i.e. no uniform neighborhood of $H$ in $G$ is 2-separating, then $H$ has interlaced cosets.
If $\tilde{e}(G,H)>1$ and if some uniform neighborhood of $H$ satisfies the noncrossing condition, then $H$ does not have interlacing cosets in $G$.

\begin{rmk}
Suppose that $H$ has interlaced cosets in $G$.
As $H$ satisfies the shallow condition, it follows that, for any $r,r'>0$ and given any pair $D, D'$ of components of $(\mathscr{C}(G)-N_r(H))$, there is a sequence $g_1, \ldots, g_n$ and components $C_i, C_i'$ as above, such that each $g_i$ meets $C_i$ and $C_i'$ outside of the $r'$--neighborhood of $N_r(H)$.
\end{rmk}


Our characterization of when certain types of splittings have three coends is the following.

\begin{thm}\label{etilde_equals_2}
Let $G$ be a finitely generated group that splits over a subgroup $H$.
Assume that the vertex groups of the splitting are finitely generated, and that for all infinite index subgroups $K$ of the inclusion(s) of $H$ into $G$, $\tilde{e}(G,K)=1$.

Suppose that the splitting is of the form $G=A*_HB$.
Then it has three coends if and only if none of the following hold:
\begin{itemize}
\item{$[A:H]=[B:H]=2$,}
\item{$[A:H]=2$, $Comm_B(H)=H$ and $H$ has interlaced cosets in $B$, or with the roles of $A$ and $B$ reversed, or}
\item{$Comm_A(H) = Comm_B(H) = H$ and $H$ has interlaced cosets in $A$ and $B$.}
\end{itemize}

Suppose instead that $G=A*_H$.
Let $t$ denote the stable letter, let $i_1, i_2 \co H \hookrightarrow A$ be the associated inclusions, and let $H_1=i_1(H)$ and $H_2=i_2(H)$.
Then the splitting has three coends if and only if none of the following hold:
\begin{itemize}
\item{$H_1=A=H_2$, or}
\item{$Comm_A(H_1)=H_1$, $Comm_A(H_2)=H_2$, $t \notin Comm_A(H_1)$ (or equivalently $t \notin Comm_A(H_2)$) and both $H_1$ and $H_2$ have interlaced cosets in $A$.}
\end{itemize}
\end{thm}

(\cite[Theorem 3.7]{Houghton} is a characterization of a splitting of $G$ over $H$ having three coends, under the assumption that $H$ has infinite index in its normalizer.)

For an example of a splitting that is one of the ``interlacing coset types'' above, and hence $\tilde{e}(G,H)=2$, consider the following.
Let $\Sigma$ be a closed surface of genus at least two, let $\gamma$ be a closed curve in $\Sigma$ that is homotopically nontrivial and has positive self-intersection number, and let $\Sigma'$ denote two copies of $\Sigma$, identified along $\gamma$.
Let $G = \pi_1(\Sigma')$, let $A$ and $B$ denote subgroups corresponding to the two copies of $\Sigma$ in $\Sigma'$ and let $H$ be an infinite cyclic subgroup that is induced by $\gamma \co S^1 \to \Sigma'$.
Then $Comm_A(H) = Comm_B(H) =H$, and $H$ has interlaced cosets in $A$ and $B$.

A similar construction is given by replacing $\Sigma$ above with a closed hyperbolic 3--manifold $M$, and letting $\gamma$ denote any closed curve in $M$ that is homotopically nontrivial.  Let $G$ be the fundamental group of two copies of $M$ identified along $\gamma$, let $A$ and $B$ denote subgroups corresponding to the two copies of $M$ and let $H\cong \mathbb{Z}$ correspond to $\gamma$.
Again $Comm_A(H) = Comm_B(H) = H$, and $\tilde{e}(A,H) = \tilde{e}(B,H)=1$, so $H$ has interlaced cosets in $A$ and $B$ here as well.

In the next proofs, we will use the following notation.
If $Y$ is a connected subset of $\mathscr{C}(G)$, $X \subseteq Y$ and $r \geq 0$, then we will write $N_r(X,Y)$ to denote the $r$--neighborhood of $X$ in $Y$, with respect to the induced path metric for $Y$.
We will let $N_r(X)$ denote the $r$--neighborhood of $X$ in $\mathscr{C}(G)$.
For any subset $Z$ of $G$, let $\overline{Z}$ denote the subgraph of $\mathscr{C}(G)$ consisting of $Z$ together with all edges that have both vertices contained in $Z$.

\begin{proof}[Proof of Theorem \ref{etilde_equals_2}]
This proof will make use of Lemma \ref{coends_and_separating} and Propositions \ref{prop1}, \ref{prop2} and \ref{prop3} below.

Assume first that $G= A*_HB$.
It suffices to take the associated generating set for $G$ to be the union of the finite generating sets associated to $A$ and $B$.
Thus for all $g \in G$ and $C \in \{A,B\}$, $g\overline{C} = \overline{(gC)}$ is simplicially isomorphic to $\mathscr{C}(C)$.

Let $T_A, T_B$ be sets of transversals for $H$ in $A$ and $B$ respectively, so $T_A$ contains exactly one representative of each coset $aH$ of $H$ in $A$, with $e$ representing $H$, and similarly for $T_B$.
It is shown in \cite{ScottWall} that each element of $G$ has a unique representation of the form
$$a_1b_2a_3b_4\cdots b_nh,$$
where $n \geq 2$, $h \in H$ and each $a_i \in T_A$ and $b_i \in T_B$, with $a_i=e$ only if $i=1$ and $b_i=e$ only if $i=n$.
In the following, when we write $a_1b_2\cdots b_n$ or $a_1b_2\cdots a_{n-1}$, it will be understood that these are subwords as the notation indicates of words in this normal form.
Hence $G$ is the union of the cosets of the form $a_1b_2a_3\cdots b_nA$ and $a_1b_2 \cdots a_{n-1}B$.

The action of $G$ on the associated Bass-Serre tree allows one to see that $A \cap B = H$.
Similarly one can see that, of all the cosets $a_1b_2a_3\cdots b_nA$ and $a_1b_2 \cdots a_{n-1}B$, $A$ meets precisely those equal to $aB$ for some $a \in T_A$, in $aH$, and $B$ meets precisely the cosets $bA$, for $b \in T_B$, in $bH$.
In addition, each $a_1b_2a_3\cdots a_{n-1}b_nA$ meets precisely the cosets $a_1b_2a_3\cdots a_{n-1}b_naB$, $a \in T_A$, in $a_1b_2a_3\cdots b_naH$, and $a_1b_2a_3\cdots a_{n-1}B$ meets precisely all $a_1b_2a_3\cdots a_{n-1}bA$ such that $b \in T_B$, in $a_1b_2a_3\cdots a_{n-1}bH$.

The Bass-Serre tree encodes all these intersections, in the following sense.
One edge of the tree has vertices say $v_A$ and $v_B$, with stabilizers $A$ and $B$ respectively, and the cosets of $A$ ($B$ respectively) translate $v_A$ ($v_B$ respectively) to the different vertices in its orbit.
Two vertices are adjacent precisely when the corresponding cosets intersect.
Since our generating set for $G$ is the union of the generating sets for $A$ and $B$, note that two vertices in $\mathscr{C}(G)$ are connected by an edge only if they are in the same coset of $A$ or $B$.

It follows from all of this that $H$ separates $A$ from $B$ in $\mathscr{C}(G)$, and similarly for any $g \in G$, $gH$ separates $gA$ from $gB$.
In particular, let $X_B = \{ a_1b_2a_3\cdots b_nh : a_1 = e, n>2\}$ and note that $X_B$ is precisely the vertex set of the component(s) of $\mathscr{C}(G)-H$ that meet $B$.
Then $X_B^* = (H \cup \{  a_1b_2a_3\cdots b_nh : a_1 \neq e\})$ is the union of $H$ with the vertices of the components that meet $A$.
It follows that $\delta X_B \subseteq N_1(H)$, so $X_B$ represents an element of $(\mathcal{P}G/\mathcal{F}_H(G))^G$.  

We claim that $X_B$ is nontrivial.
For it follows from \cite{ScottWall} that each $g \in G$ has a unique ``reversed'' normal form $ha_1b_2\cdots b_n$ for $h \in H$ and $a_1b_2\cdots b_n$ as above.
Hence the elements of the form $a_1b_2a_3\cdots b_n$ are contained in distinct cosets $Hg$, and so both $X_B$ and its complement are $H$--infinite.
Thus $X_B$ is nontrivial and so we recover that $\tilde{e}(G,H) \geq 2$.
\\

Our argument below will be ordered as follows, and we note that all results mentioned will be symmetric in $A$ and $B$.
We will begin by noting the standard result that $\tilde{e}(G,H)=2$ if $[A:H]=[B:H]=2$ (thus $\tilde{e}(A,H)=\tilde{e}(B,H)=0$).
If $3 \leq [A:H] <\infty$ (so again $\tilde{e}(A,H)=0$), then we will see that $\tilde{e}(G,H) \geq 3$.
Next, we will show that $\tilde{e}(G,H) \geq 3$ if $\tilde{e}(A,H)>0$ and $Comm_A(H) \neq H$, or if $H$ does not have interlaced cosets in $A$ (hence $\tilde{e}(A,H)>1$).
We will refer to Proposition \ref{prop1} below to see that $\tilde{e}(G,H)=2$ if $\tilde{e}(A,H)=0$ with $[A:H]=2$, $Comm_B(H)=H$ and $H$ has interlaced cosets in $B$, and finally Proposition \ref{prop2} below gives $\tilde{e}(G,H)=2$ if $Comm_A(H) = Comm_B(H) = H$ and $H$ has interlaced cosets in $A$ and $B$.
This will exhaust all possible cases when $G = A*_HB$.

If $[A:H]=[B:H]=2$ (so $\tilde{e}(A,H) = \tilde{e}(B,H) = 0$), then it is well-known that $\tilde{e}(G,H)=2$.  
Indeed, it follows that $H$ is normal in $A$ and $B$, hence in $G$, and so $G/H \cong \mathbb{Z}_2*\mathbb{Z}_2$, which is two-ended.
It is shown in \cite{KrophollerRoller} that if $H$ is normal in $G$, then $\tilde{e}(G,H) = e(G,H) = e(G/H)$, thus $\tilde{e}(G,H) = 2$ in this case.

For our next case, we consider $\tilde{e}(A,H)=0$ with $[A:H] \geq 3$, and let $e,a,a' \in T_A$ be distinct.
Recall that $\tilde{e}(A,H)=0$ if and only if $[A:H]<\infty$, or equivalently $A \subseteq N_r(H)$ for some $r \geq 0$.
Fix some such $r$ and it follows that $N_r(H)$ contains $aH$ and $a'H$.
The unions of the components of $(\mathscr{C}(G)-A)$ that meet each of $B, aB$ and $a'B$ are $X_B, aX_B$ and $a'X_B$ respectively hence they are not contained in any uniform neighborhood of $H$.
As $H$ satisfies the shallow condition, the complement of $N_r(H)$ has a deep component contained in each of these, and hence $\tilde{e}(G,H) \geq 3$.
We similarly get this conclusion if $\tilde{e}(B,H)=0$ and $[B:H] \geq 3$.

Suppose next that $\tilde{e}(A,H)>0$, i.e. $[A:H]=\infty$, and suppose that $Comm_A(H)\neq H$.
Let $a \in Comm_A(H) \cap (T_A-\{ e\})$, and let $r >0$ be such that $aH \subseteq N_r(H)$.
As we saw in the previous case, there are two deep components of $(\mathscr{C}(G)-N_r(H))$ that are contained in $X_B$ and $aX_B$ respectively, and neither meets $A$.
As $[A:H]=\infty$, $A$ contains infinitely many cosets $Ha'$ of $H$.
Since $H$ satisfies the shallow condition, it follows that $A$ meets another deep component of the complement of $N_r(H)$.
So $N_r(H)$ is 3-separating and $\tilde{e}(G,H) \geq 3$ in this case too, and similarly if $\tilde{e}(B,H)>0$ and $Comm_B(H)\neq H$.

If $H$ does not have interlacing cosets in $A$, so $\tilde{e}(G,H)>1$, then there is some $r \geq 0$ such that $(\overline{A}-N_{r}(H, \overline{A}))$ has two deep components that are not connected by a sequence of cosets of $H$, as in the definition of interlacing cosets.
Thus the components of $(\overline{A}-N_{r}(H, \overline{A}))$ can be partitioned into two sets, $\mathscr{D}_1$ and $\mathscr{D}_2$, such that each contains at least one deep component, and no coset $aH$ meets components in both $\mathscr{D}_1$ and $\mathscr{D}_2$.
Note that there is an $H$--infinite component $C_i$ of $(\mathscr{C}(G)-N_r(H))$ that meets each $\mathscr{D}_i$, and it follows from the definition of the $\mathscr{D}_i$'s and the ``tree'' configuration of the cosets of $A,B$ and $H$ in $G$ that $C_1 \cap C_2 = \emptyset$.
Moreover, $B$ meets an $H$--infinite component of $(\mathscr{C}(G)-H)$ that is disjoint from $A$, and this component must contain a deep component of the complement of $N_r(H)$, that is disjoint from $C_1$ and $C_2$.
It follows that $N_r(H)$ is 3-separating, so $\tilde{e}(A,H) \geq 3$.

By Proposition \ref{prop1} below, we have that if $[A:H]=2$ (so $\tilde{e}(A,H)=0$), $Comm_B(H)=H$ and $H$ has interlacing cosets in $B$, then $\tilde{e}(G,H)=2$, and similarly if we exchange the roles of $A$ and $B$.

Finally, Proposition \ref{prop2} shows that $\tilde{e}(G,H)=2$ if $Comm_A(H) = Comm_B(H) = H$ and $H$ has interlaced cosets in $A$ and $B$.
This completes the proof of the theorem in the case that $G=A*_HB$.
\\

Now let us consider the HNN extension case, so
$$G=A*_H = (A*\langle t \rangle ) / N,$$
where $N$ denotes the normal closure of $\{ i_2(h)^{-1}t^{-1}i_1(h)t : h \in H\}$ in $(A *\langle t\rangle)$.
Let $H_k = i_k(H)$ for $k=1,2$, so $H_2 =t^{-1}H_1t $ and hence $\tilde{e}(G,H_1) = \tilde{e}(G,H_2)$.
Thus the splitting $G = A*_H$ has three coends if and only if $\tilde{e}(G,H_1)=\tilde{e}(G,H_2) \geq 3$.
The argument that follows will differ from the amalgamated free product case in subtle ways.

We shall take the generating set for $G$ to be the union of $t$ with a finite generating set for $A$.
Thus $\mathscr{C}(A)\cong \overline{A} \cong g\overline{A}=\overline{(gA)}$ for any $g \in G$.
Let $T_k$ be a set of transversals for $H_k$ in $A$, so $T_k$ contains exactly one representative of each coset $aH_k$ of $H_k$ in $A$, and we take $e$ to represent $H$.

Scott and Wall \cite{ScottWall} show that any $g \in G$ has a unique normal form 
$$a_1t^{\epsilon_1}a_2t^{\epsilon_2}\cdots a_nt^{\epsilon_n}a_{n+1},$$
where $n \in \mathbb{Z}_{\geq 0}$, $a_{n+1} \in A$, and for all $k\leq n$, $\epsilon_k = \pm 1$, $a_k \in T_1$ if $\epsilon_k=1$, $a_k \in T_2$ if $\epsilon_k=-1$, and $a_k \neq e$ if $k>1$ and $\epsilon_{k-1} \neq \epsilon_k$.
When we write $a_1t^{\epsilon_1}\cdots a_nt^{\epsilon_n}a_{n+1}$ or $a_1t^{\epsilon_1}\cdots a_nt^{\epsilon_n}$ in the following, we shall assume that the words are of this form.
Thus $G$ is the disjoint union of the cosets of $A$ of the form $a_1t^{\epsilon_1}\cdots a_nt^{\epsilon_n}A$.

Let us consider which $a_1t^{\epsilon_1}\cdots a_nt^{\epsilon_n}$ are such that $\overline{A}$ is connected to $a_1t^{\epsilon_1}\cdots a_nt^{\epsilon_n}\overline{A}$ by an edge.
The edges meeting $A$ are of the form $[a,as]$, for $a \in A$ and $s \in A$ or $s=t^{\pm 1}$.
If $s \in A$ then $[a,as] \subseteq \overline{A}$.
If $s=t^{-1}$ then write $a=a'h$ for $a' \in T_2$ and $h\in H_2$, and then $at^{-1} = a'ht^{-1} = a't^{-1}h'$, where $h' \in H_1$.
Note that $a't^{-1}h'$ is in normal form, so $a'H_2 \subseteq \overline{A}$ is connected to $a't^{-1}H_1\subseteq a't^{-1}\overline{A}$ by an edge in $\mathscr{C}(G)$, for any $a' \in T_2$.
If $s=t$ then write $a = a'h$ for $a' \in T_1$ and $h \in H_1$, and similarly $at = a'th'$ for $h' \in H_2$, which is also the normal form.  
Thus $a'H_1 \subseteq \overline{A}$ is also connected to $a'tH_2 \subseteq a't\overline{A}$ by an edge, for any $a' \in T_1$.
Hence $\overline{A}$ is connected by single edges precisely to $a't^{-1}\overline{A}$, $a' \in T_2$, and $a't\overline{A}$, $a' \in T_1$.

Similarly for any $g \in G$, $g\overline{A}$ is connected by single edges to only $ga't^{-1}H_1 \subseteq ga't^{-1}\overline{A}$ through $ga'H_2$ for any $a' \in T_2$ and $gatH_1 \subseteq gat\overline{A}$ through $ga'H_1$ for any $a' \in T_1$.

Note that the Bass-Serre tree associated to $G=A*_H$ encodes this adjacency information, similarly to the amalgamated free product case.
In particular, for any $g \in G$, $N_1(gH_1)$ contains $gH_1$ and $gH_1t = gtH_2$, hence separates $g\overline{A}$ and $gt\overline{A}$ from one another.

Let $X_1 := \{ a_1t^{\epsilon_1}\cdots a_nt^{\epsilon_n}a_{n+1} : a_1=e, \epsilon_1=1\}$.
It follows from our discussion that $\delta X_1 \subseteq N_1(H_1)$, so $X_1$ represents an element of $(\mathcal{P}G/\mathcal{F}H_1)^G$.

Consider any $g=a_1t^{\epsilon_1}\cdots a_nt^{\epsilon_n}a_{n+1}$, and note that if we use transversals $T_k'$ of $H_k\backslash A$, then ``pushing'' elements of $H_1, H_2$ to the left in this word uniquely determines a reversed normal form $g=ha_1't^{\epsilon_1}\cdots a_n't^{\epsilon_n}a_{n+1}'$, where $a_k \in T_2'$ if $\epsilon_{k-1}=1$ and $a_k \in T_1'$ if $\epsilon_{k-1} =-1$.
Hence if the exponents $\epsilon_k$ for two different elements of $G$ do not match, then the elements are in different cosets $H_1(g')$.
In particular, it follows that $X_1$ and its complement are $H_1$--infinite, and so represent nontrivial elements of $(\mathcal{P}G/\mathcal{F}H_1)^G$.
Thus we recover the conclusion of Proposition \ref{splitting_coends} in this case:  $\tilde{e}(G,H_1) \geq 2$.
\\

Our argument will develop as follows.
The statements we give are symmetric in $H_1$ and $H_2$.
First, we will give the standard result that $\tilde{e}(G,H_1)=2$ if $A = H_1 = H_2$.
Next we will show that $\tilde{e}(G,H_1) \geq 3$ if $\tilde{e}(A,H_1)=0$ and $A \neq H_1$.
We shall see that if $A=H_1$, then $\tilde{e}(G,H_2) = 0$, which puts us in one of the two previous cases.
Then we have that $\tilde{e}(G,H_1) \geq 3$ if $\tilde{e}(A,H_1)>0$ and $Comm_A(H_1) \neq H_1$.
The same result follows if $H_1$ does not have interlaced cosets in $A$.
If $Comm_A(H_1) = H_1$, $Comm_A(H_2) = H_2$, $t \notin Comm_G(H_1)$ and both $H_1$ and $H_2$ have interlaced cosets in $A$, then Proposition \ref{prop3} below gives $\tilde{e}(G,H_1) = 2$.
If instead $Comm_A(H_1) = H_1$, $Comm_A(H_2) = H_2$ and $H_1$ and $H_2$ have interlaced cosets in $A$, but $t \in Comm_G(H_1)$ then it follows that $\tilde{e}(G,H_1) \geq 3$.

If $A=H_1=H_2$ then $\tilde{e}(G,H_1)=2$, since in this case $H_1 \lhd G$, with $G/H_1 \cong \mathbb{Z}$ so by \cite{KrophollerRoller}, $\tilde{e}(G,H_1) = e(G/H_1) = e(\mathbb{Z}) = 2$.

Suppose that $\tilde{e}(A,H_1)=0$ and $A \neq H_1$.  
Then there is some nontrivial $a \in T_1$, and note that $tA$, $atA$ and for instance $t^{-1}A$ are all contained in distinct components of $(\mathscr{C}(G)-\overline{A})$.
Each of these components is $H_1$--infinite, for the component containing $tA$ has vertex set $X_1$ defined earlier, and the component containing $atA$ has vertex set $aX_1$, hence contains words with normal forms containing arbitrarily large numbers of ``$t$'s'' and so is $H_1$--infinite by the observation made above.
Finally, the component containing $t^{-1}A$, call it $Y$, also contains all words that have normal form $a_1t^{\epsilon_1}\cdots a_nt^{\epsilon_n}a_{n+1}$ such that $a_1=e$ and $\epsilon_1=-1$, so this set must also be $H_1$--infinite, by the same argument.
Since $\tilde{e}(A,H)=0$, there is some $r>0$ such that $A \subseteq N_r(H_1)$.
Also $H_1$ satisfies the shallow condition in $\mathscr{C}(G)$, and it follows that each of $X_1, aX_1$ and $Y$ contains a deep component of $(\mathscr{C}(G)-N_r(H_1))$.
Thus $N_r(H_1)$ is 3-separating and hence $\tilde{e}(G,H_1) \geq 3$.
Similarly $\tilde{e}(G,H_2) \geq 3$ if $\tilde{e}(A,H_2)=0$ and $A \neq H_2$.

Note that if $A=H_1$ then $H_2\subseteq H_1$.
As $d_{Haus}(H_2=t^{-1}H_1t, t^{-1}H_1)<\infty$, it follows that $t^{-1}H_1$ is contained in some uniform neighborhood of $H_1$.
By Lemma \ref{deep_for_subgroups}, $H_1$ and its translates satisfy the deep condition, so Lemma \ref{cis} implies then that $d_{Haus}(H_1, t^{-1}H_1)<\infty$.
It follows that $d_{Haus}(H_1, H_2)<\infty$, so $[A:H_2]<\infty$.
This puts us in one of the previous cases, and similarly for $A=H_2$.

If $\tilde{e}(A,H_1)>0$ and $Comm_A(H_1) \neq H_1$, then let $a \in Comm_A(H_1)\cap (T_1-\{e\})$, and, as in the amalgamated free product case, we consider a neighborhood $N_r(H_1)$ that contains $aH_1$.
Then $N_r(H_1)$ must separate regions of $\mathscr{C}(G)$ that meet $X_1$, $aX_1$ and $A$.
Since also $[A:H_1]=\infty$ and $H_1$ satisfies the shallow condition, it follows that each of $X_1, aX_1$ and $A$ must meet a deep component of $(\mathscr{C}(G)-N_r(H_1))$, thus $\tilde{e}(G,H_1)\geq 3$.
Naturally the analogous result holds if $H_1$ is replaced by $H_2$.

Suppose that $[A,H_1]=\infty$ and that $H_1$ does not have interlacing cosets in $A$, so in particular $\tilde{e}(A,H_1)>1$.
As in the amalgamated free product case, we can take a neighborhood $N_r(H_1, \overline{A})$ whose complement in $\overline{A}$ contains deep components $D_1, D_2$ that are not connected by a sequence of cosets of $H_1$.
It follows that $(\mathscr{C}(G)-N_r(H_1, \overline{A}))$ contains distinct $H_1$--infinite components that meet $D_1, D_2$ and $tA$, in fact $X_1$, respectively.
Thus $N_r(H_1, \overline{A})$ is contained in a 3-separating uniform neighborhood of $H_1$ in $\mathscr{C}(G)$, and it follows that $\tilde{e}(G,H_1) \geq 3$.
Similarly we get $\tilde{e}(G,H_2) \geq 3$ if $[A,H_2]=\infty$ and $H_2$ does not have interlacing cosets in $A$

It remains to consider the case that $Comm_A(H_1) = H_1$, $Comm_A(H_2)=H_2$ and both $H_1$ and $H_2$ have interlaced cosets in $A$.
Recall that $t \notin A$, so we can consider whether or not $t \in Comm_G(H_1)$.
If $t \notin Comm_G(H_1)$ then Proposition \ref{prop3} below shows that $\tilde{e}(G,H_1)=2$.
If $t \in Comm_G(H_1)$, then $d_{Haus}(H_1, H_2)<\infty$.
Let $r>0$ be such that $H_2 \subseteq N_r(H_1)$, and as in the case that $\tilde{e}(A,H_1)=0$ and $A \neq H_1$, we have that $(\mathscr{C}(G)-N_r(H_1))$ contains a deep component that meets $X_1$, and a deep component that meets $\{ a_1t^{\epsilon_1}\cdots a_{n+1} : a_1=e, \epsilon_1=-1\}$, and neither of these meets $A$.
In addition we have that $[A:H_1]=\infty$, so $(\mathscr{C}(G)-N_r(H_1))$ contains a deep component that has $H_1$--infinite intersection with $A$, hence $\tilde{e}(G,H_1)\geq 3$.
\end{proof}

We now give Propositions \ref{prop1}, \ref{prop2} and \ref{prop3}, which were used in the previous proof.
We give a careful proof of Proposition \ref{prop1} below.  
The other situations are morally the same, and can be proved using similar methods.

\begin{prop}\label{prop1}
Suppose that $G=A*_HB$, with $A$ and $B$ finitely generated, and that, for every infinite index subgroup $K$ of $H$, $\tilde{e}(G,K) = 1$.
If $[A:H]=2$, $Comm_B(H)=H$ and $H$ has interlaced cosets in $B$, then $\tilde{e}(G,H)=2$.
\end{prop}

\begin{proof}
Let $T_A = \{ e,a\}$ and $T_B$ be transversals for $H$ in $A$ and $B$ respectively.
As above, it suffices to assume that the finite generating set for $G$ is a union of finite generating sets for $A$ and $B$, and take the generating set for $A$ to be $a$ together with a finite generating set for $H$.
Thus we have $\mathscr{C}(A) \cong \overline{A} \cong g\overline{A}$ and $\mathscr{C}(B) \cong \overline{B} \cong g\overline{B}$ for all $g \in G$.

Recall the subset $X_B$ of $G$ from the proof of Theorem \ref{etilde_equals_2}, which represents a nontrivial element of $(\mathcal{P}G/\mathcal{F}_H(G))^G$, and suppose that $X$ also represents an element of $(\mathcal{P}G/\mathcal{F}_H(G))^G$.
We will show that $X_B \stackrel{H}{\subseteq} X$ or $(X \cap X_B)$ is $H$--finite, and a similar argument shows the analogous result for $X_B^*$.
It follows that $X$ is $H$-almost equal to $X_B$, $X_B^*$, $G$ or $\emptyset$.
Hence $\{ X_B, G\}$ generates $(\mathcal{P}G/\mathcal{F}_H(G))^G$ over $\mathbb{Z}_2$, so $\tilde{e}(G,H) = 2$.

We have that the identity map $(B, d_B) \hookrightarrow (B,d_G)$ is a $(\phi, \Phi)$-uniformly distorting map, for some $\phi$ and $\Phi$.
Our choice of generating sets gives $\mathscr{C}(B) \cong \overline{B}$, so we can take $\Phi$ to be the identity.
For any $r \geq 0$, we have $N_{r}(H,\overline{B}) \subseteq (N_{r}(H) \cap \overline{B})$.
On the other hand, let $\phi' \co \mathbb{R}_{\geq 0} \to \mathbb{R}_{\geq 0}$ be such that $\phi (\phi'(r))>r$, and it follows that, for any $r \geq 0$, $(N_r(H)\cap \overline{B}) \subseteq N_{\phi'(r)}(H, \overline{B})$.
Thus $\phi'(r) \geq r$ for all $r$.
We also have $(N_r(gH)\cap g\overline{B}) \subseteq N_{\phi'(r)}(gH, g\overline{B})$ and $N_{r}(gH,g\overline{B}) \subseteq (N_{r}(gH) \cap g\overline{B})$ for any $g \in G$.

As $X$ is a representative of an element of $(\mathcal{P}G/\mathcal{F}_H(G))^G$, $\delta X$ is $H$--finite, so there is some $r \geq 0$ such that $\delta X \subseteq N_r(H)$ and $(\delta X \cap \overline{B}) \subseteq N_{\phi'(r)}(H,\overline{B})$.
Thus each component of $(\overline{B}-N_{\phi'(r)}(H, \overline{B}))$ is either contained in $\overline{X}$ or is contained in $\overline{X}^*$.

We claim that the deep components of $(\overline{B}-N_{\phi'(r)}(H, \overline{B}))$ are all contained in $\overline{X}$, or are all contained in $\overline{X}^*$.
Since $H$ has interlacing cosets in $B$, it suffices to show that, for a fixed constant $\rho$, any pair of components that meet the same coset $bH$ outside of $N_\rho(H)$ must both be in $\overline{X}$ or both be in $\overline{X}^*$.

Suppose first that $r <2$, and that $bH$ meets components $C,C'$ of $(\overline{B}-N_{\phi'(r)}(H, \overline{B}))$ outside of $N_{r+1}(H)$, in points $g_1$ and $g_2$ respectively.
Then $g_ia \in bHa = baH$, and $g_ia$ is connected to $g_i$ by an edge that is not in $\delta X\subseteq N_r(H)$ for each $i$.

Recall that $\overline{B}$ and $ba\overline{B}$ are disjoint and separated by $bA = (bH \coprod baH)$, hence any path from $H$ to $baH$ (or to $baB$) must pass through $bH$.
It follows that any pair of points in $H$ and $baH$ respectively are of distance at least two from one another.
Hence $N_r(H)$ is disjoint from $baH$, and so $\delta X$ is disjoint from $ba\overline{B}$.
Therefore $ba\overline{B}$ is contained in $\overline{X}$ or $\overline{X}^*$.
Since the edges $[g_i, g_ia]$ connect $C, C'$ to $ba\overline{B}$ without meeting $\delta X$, it follows that $ba\overline{B} \subseteq \overline{X}$ implies that $C\cup C' \subset \overline{X}$, and $ba\overline{B} \subseteq \overline{X}^*$ implies that $C\cup C' \subset \overline{X}^*$.
Hence the deep components of $(\overline{B}-N_{\phi'(r)}(H, \overline{B}))$ are all contained in $\overline{X}$ or are all in $\overline{X}^*$ in this case, as desired.

Suppose instead that $r \geq 2$.
Let $\rho = (r+m_0(\phi'(r-2))+1)$ and suppose that $bH$ meets components $C$ and $C'$ of $(\overline{B}-N_{\phi'(r)}(H, \overline{B}))$ outside of $N_{\rho}(H)$, in points $g_1$ and $g_2$ respectively.
Thus, for each $i$, $B_{m_0(\phi'(r-2))+1}(g_i)$ is disjoint from $\delta X \subseteq N_r(H)$.
We have that $g_ia \in baH$ is of distance 1 from $g_i$, so $B_{m_0(\phi'(r-2))}(g_ia)$ also is disjoint from $\delta X$.
Using that $H$ satisfies $deep(m_0)$, it follows that there is a path from each $g_i$ to any deep component $D$ of $(ba\overline{B}-N_{\phi'(r-2)}(baH))$, that does not meet $\delta X$.

Recall that any point of $baB$ is of distance at least two from $H$.
Hence $(\delta X \cap ba\overline{B}) \subseteq (N_{r-2}(baH)\cap ba\overline{B}) \subseteq N_{\phi'(r-2)}(baH, ba\overline{B})$.
Thus $D$ does not meet $\delta X$, and it follows that $g_1$ can be connected to $g_2$ by a path that does not meet $\delta X$.
Therefore $C \subseteq \overline{X}$ if and only if $C' \subseteq \overline{X}$, so since $H$ has interlacing cosets in $B$, the deep components of $(\overline{B}-N_{\phi'(r)}(H, \overline{B}))$ are all contained in $\overline{X}$, or are all contained in $\overline{X}^*$, as claimed.
We shall call this argument (*), for future reference.

Thus, for every $r >0$, the deep components of $(\overline{B}-N_{\phi'(r)}(H, \overline{B}))$ are all contained in $\overline{X}$ or are all contained in $\overline{X}^*$; assume that they are all contained in $\overline{X}$.
We will use this to show that $(X^* \cap X_B)$ is $H$--finite, i.e. that $X_B \stackrel{H}{\subseteq} X$.
By considering $X^*$ instead, it follows that if all deep components are instead contained in $\overline{X}^*$ then $(X \cap X_B)$ is $H$--finite, as desired.

All deep components of $(\overline{B}-N_{\phi'(r)}(H, \overline{B}))$ are contained in $\overline{X}$, so $(\overline{X}^* \cap \overline{B})$ is contained in the union of $N_{\phi'(r)}(H, \overline{B})$ with the shallow complementary components in $\overline{B}$, i.e. for $r_0= (\phi'(r)+m_1(\phi'(r)))$, we have $(\overline{X}^* \cap \overline{B}) \subseteq N_{r_0}(H, \overline{B}) \subseteq N_{r_0}(H)$.

Let $b_1 \in T_B-\{ e\}$ and consider $b_1H$.
Recall that $Comm_B(H)=H$, hence $b_1 \notin Comm_B(H)$.
Note that $b_1H$ is 2--separating, and by Lemma \ref{deep_for_subgroups}, both $H$ and $b_1H$ satisfy the deep condition.
Thus it follows from Lemma \ref{cis} that $b_1H$ is not contained in any uniform neighborhood of $H$.
Hence for any $\rho' \geq 0$, there is some $g \in b_1H$ such that $B_{\rho'} (g) \cap H = \emptyset$.

Note that if $r<2$, then $\delta X$ does not meet $b_1a\overline{B}$, so a simplified version of our earlier argument shows that $b_1a\overline{B}\subseteq \overline{X}$.

Assume instead that $r \geq 2$.
Part of the argument (*) shows that $(\delta X \cap b_1a\overline{B}) \subseteq N_{\phi'(r-2)}(b_1aH, b_1a\overline{B})$, hence each component of $(b_1a\overline{B}-N_{\phi'(r-2)}(b_1aH, b_1a\overline{B}))$ is entirely contained in $\overline{X}$ or $\overline{X}^*$.
Let $\rho' > \max\{ r_0, r+m_0(\phi'(r-2))+1\}$ and let $g \in b_1H$ be such that $B_{\rho'}(g) \cap H = \emptyset$.
As $g \in B$, $g \notin N_{r_0}(H)$, we have that $g \in X$.
As $\rho'>(r+1)$, $ga \notin \delta X \subseteq N_r(H)$, so the edge $[g,ga]$ is not contained in $\delta X$, and it follows that $ga \in X$.
As $\rho' > (r+m_0(\phi'(r-2))+1)$, $B_{m_0(\phi'(r-2))}(ga) \cap N_{r}(H) = \emptyset$, and it follows that $B_{m_0(\phi'(r-2))}(ga) \subseteq \overline{X}$.
This ball meets all deep components of $(b_1a\overline{B}-N_{\phi'(r-2)}(b_1aH, b_1a\overline{B}))$, so they all must be contained in $\overline{X}$.

Hence $(X^* \cap b_1aB)$ is contained in the union of $N_{\phi'(r-2)}(b_1aH, b_1a\overline{B})$ together with the shallow components of its complement, i.e. we have $(X^* \cap b_1a\overline{B}) \subseteq N_{\phi'(r-2)+m_1(\phi'(r-2))}(b_1aH, b_1a\overline{B})$.

In fact we can say more.
We saw above that if $\rho>0$, $g \in b_1H$ and $B_\rho (g) \cap H = \emptyset$, then $ga \in (b_1aH \cap X)$.
It follows that, for any $ga \in b_1aH$, if $B_{\rho+1}(ga) \cap H = \emptyset$, then $ga \in X$.
In other words, $(b_1aH-N_{\rho+1}(H)) \subseteq X$, and thus $(\overline{X}^*\cap b_1aH) \subseteq N_{\rho+1}(H)$.

Recall that any component of $(b_1a\overline{B}-N_{\phi'(r-2)}(b_1aH, b_1a\overline{B}))$ is entirely contained in $\overline{X}$ or in $\overline{X}^*$.
Let $S$ be a shallow component, and suppose that $S \subseteq \overline{X}^*$.
Let $fr(S, b_1a\overline{B})$ denote the frontier of $S$ in $b_1a\overline{B}$, and let $fr(S)$ denote the frontier of $S$ in $\mathscr{C}(G)$.
Then $fr(S, b_1a\overline{B}) \subseteq N_{\phi'(r-2)}(b_1aH)$, so any point $p \in fr(S)$ can be connected to some $p' \in b_1aH$ by a path of length $\leq \phi'(r-2)$.
If $p' \in \overline{X}^*$, then $p \in N_{\phi'(r-2)+\rho+1}(H)$.
If $p' \in \overline{X}$, then this path must meet $\delta X$, so $p \in N_{r+\phi'(r-2)}(H)$.
Thus, if $\rho'' > (\phi'(r-2)+\rho'+1) > (r+\phi'(r-2))$, then we have that $fr(S, b_1a\overline{B}) \subseteq N_{\rho''}(H)$.
Let $r_1 =  (\rho''+m_1(\phi'(r-2)))$ and it follows that $S \subseteq N_{r_1}(H)$.

As $r_1$ does not depend on our choice of $S$, we have $(X^* \cap b_1aB) \subseteq N_{r_1}(H)$, so in particular $(X^* \cap b_1aB)$ is $H$--finite.
Also $r_1$ does not depend on our choice of $b_1 \in T_B-\{ e\}$, so it follows that 
$$\left( X^* \cap \bigcup_{b_1\in T_B-\{ e\} } b_1aB\right) \subseteq N_{r_1}(H),$$
and therefore this intersection is $H$--finite.

As for the cosets in $X_B$ of the form $b_2ab_1aB$, note that if we translate everything in the above argument by $(b_2a)$ and replace $r$ with $(r-2)$, then we can see that $(\delta X \cap b_2ab_1a\overline{B}) \subseteq N_{r-4}(b_2ab_1aH)$, and we also get that $(X^* \cap b_2ab_1aB)$ is contained in some uniform neighborhood of $H$ that is independent of our choice of $b_2 \in T_B-\{ e\}$.
It follows that 
$$ X^* \cap \bigcup b_2ab_1aB$$
is also $H$--finite, where the union is taken over all choices of $b_1, b_2 \in T_B-\{ e\}$.

Continuing in this manner, we see that 
$$(\delta X \cap b_nab_{n-1}a\cdots b_1a\overline{B}) \subseteq N_{r-2n}(b_na\cdots b_1aH),$$ 
so for all $n> r/2$, $\delta X$ does not meet any coset of the form $b_na\cdots b_1aB$ (nor any coset $b_{n+1}ab_na\cdots b_1A$).
An analogous argument to the discussion of $r<2$ above gives that these cosets are entirely contained in $\overline{X}$.
Also we get that for every $n\leq r/2$,
$$X^* \cap \bigcup  b_na\cdots b_1aB$$ 
is $H$-finite, where the union is taken over all choices $b_1,\ldots, b_n\in T_B-\{ e\}$.
Note that $X_B$ is the union of the cosets of the form $b_na\ldots b_1aB$, and hence $X^* \cap X_B$ is $H$-finite, as desired.

\end{proof}

Similar methods prove the following.

\begin{prop}\label{prop2}
Suppose that $G=A*_HB$ for $A$, $B$ finitely generated.
If $Comm_A(H)=H=Comm_B(H)$ and $H$ has interlaced cosets in $A$ and $B$, then $\tilde{e}(G,H)=2$.
\end{prop}

\begin{prop}\label{prop3}
Suppose $G=A*_H$, with $A$ finitely generated and notation from the proof of Theorem \ref{etilde_equals_2}.
If $Comm_A(H_1)=H_1$, $Comm_A(H_2)=H_2$, $t \notin Comm_G(H_1)$, and both $H_1$ and $H_2$ have interlaced cosets in $A$, then $\tilde{e}(G,H_1)=2$.
\end{prop}


Now we will show that, under suitable hypotheses, splittings with three coends are detectable from the coarse geometry of a group, and hence are invariant under quasi-isometries.
Our results will rely on \cite{DunSwen}.

Recall that two subgroups are said to be commensurable if their intersection is of finite index in each.

\begin{thm}\label{splitting_thm0}
Let $G$ and $H$ be as in Theorem \ref{subgroup_thm}, and suppose that $H$ is finitely generated.
Then $G$ admits a splitting over a subgroup commensurable with $H$, which has three coends.
\end{thm}

In Lemma \ref{deep_for_subgroups} we showed that if, for all infinite index subgroups $K$ of $H$, $\tilde{e}(G,K)=1$, then $H$ satisfies the deep condition.
Towards the proof of Theorem \ref{splitting_thm0}, we have the following, the proof of which shows the converse to this.

\begin{lem}\label{spl_lem1}
Let $G$ and $H$ be as in Theorem \ref{subgroup_thm}.  
Then, for all infinite index subgroups $K$ of $H$, $e(G,K)=1$.
\end{lem}

\begin{proof}
Let $G,H$ and $Y$ be as in Theorem \ref{subgroup_thm}, so $Y$ satisfies the deep, shallow, 3--separating and noncrossing conditions, and $d_{Haus}(Y,H)<\infty$.
By Lemma \ref{shallow_for_subgroups}, $H$ satisfies the shallow condition.
Note that there is a quasi-isometry of $\mathscr{C}(G)$ to itself that takes $Y$ to $H$, and hence by Lemma \ref{deep_qi_inv}, $H$ satisfies the deep condition --- say $deep(m_0')$.

Note that $e(G,K)=0$ if and only if $[G:K]<\infty$, so if $K$ is an infinite index subgroup of $H$ then $e(G,K)\geq 1$.
As $e(G,K) \leq \tilde{e}(G,K)$ for any subgroup $K$ of $G$, the lemma will follow if we can show that $\tilde{e}(G,K) \leq 1$ for any infinite index subgroup $K$ of $H$.
Suppose for a contradiction that $H$ contains an infinite index subgroup $K$ such that $\tilde{e}(G,K) > 1$.
Then Lemma \ref{coends_and_separating} implies that there is some $R>0$ such that $N_R(K)$ is 2--separating.  

Next, we follow the argument from the proof of Lemma \ref{cis}.
Suppose in addition that there are two components of the complement of $N_R(K)$, say $C_1$ and $C_2$, such that $H$ meets each $C_i$ in a point $p_i$ that is not contained in the $m_0'(R)$--neighborhood of $N_R(K)$. 
So the $m_0'(R)$--ball about $p_i$ contained in $C_i$ for each $i$.

Also we have $N_R(H) \supseteq N_R(K)$, so the components of the complement of $N_R(H)$ are contained in the components of the complement of $N_R(K)$.  
In particular, $N_R(H)$ must have a deep complementary component that is disjoint from $C_1$ or is disjoint from $C_2$, and hence does not meet the $m_0'(R)$--ball about $p_1$ or $p_2$.
But this contradicts that $H$ satisfies $deep(m_0')$.

Thus $H$ minus the $m_0'(R)$--neighborhood of $N_R(K)$ must be contained in a single component of the complement of $N_R(K)$, say $C_1$.
It follows that $H$ is contained in the $m_0'(R)$--neighborhood of $N_R(K) \cup C_1$.
As $N_R(K)$ is 2--separating, there is a deep complementary component $C_2$ that meets $H$ only in the $m_0'(R)$--neighborhood of $N_R(K)$.
Recall that $H$ satisfies the shallow condition, so there must be a deep component $D$ of the complement of the $(R+m_0'(R))$--neighborhood of $H$ that is contained in $C_2$.  
In particular, note that $fr(D)$ is contained in a uniform neighborhood of $K$.

We note that $K$ is an infinite index subgroup of $H$ if and only if $H$ is not contained in any uniform neighborhood of $K$.  
Thus $H$ is not contained in any uniform neighborhood of $fr(D)$.
But this contradicts that $H$ satisfies the deep condition.
Hence, for any infinite index subgroup $K$ of $H$, $\tilde{e}(G,K)\leq 1$, and therefore $e(G,K)=1$.
\end{proof}

As an immediate corollary to Lemma \ref{spl_lem1}, we have:

\begin{cor}\label{spl_cor}
Let $G$ and $H$ be as in Theorem \ref{subgroup_thm}.  
Then for any subgroup $H_1$ of $H$ and any infinite index subgroup $K$ of $H_1$, $e(G,K)=1$.
\end{cor}

We will also need the following.

\begin{lem}\label{spl_lem2}
Let $G$ and $H$ be as in Theorem \ref{subgroup_thm}.  
Then there is a finite index subgroup $H_1$ of $H$ such that there is a nontrivial $H_1$--almost invariant set $B$ with $BH_1=B$.
\end{lem}

For the proof of this lemma, we roughly follow the proof of Proposition 3.1 in \cite{DunSwen}.

\begin{proof}
Let $G$, $H$ and $Y$ be as in Theorem \ref{subgroup_thm}, and recall that $d_{Haus}(Y,H)<\infty$.
As $Y$ is 3--separating and satisfies the noncrossing condition, it follows, in particular, that there are some $r,k=k(r) \geq 0$ such that $N=N_r(H)$ has at least two deep complementary components, and satisfies $noncrossing(k)$.
We saw in the proof of Lemma \ref{spl_lem1} that $H$, hence $N$, satisfies the deep condition.
In particular, it follows that $N$ has only finitely many deep complementary components, so there is a finite index subgroup $H_1$ of $H$ that stabilizes each deep component.

Let $U$ be one of the deep components of $(\mathscr{C}(G)-N)$, and let $U^* = (\mathscr{C}(G)-U)$.
Thus, for any $g \in G$, $gN$ is contained in the $k$--neighborhood of $U$ or the $k$--neighborhood of $U^*$.
Let $B = \{ g \in G : gN \subseteq N_k(U)\}$.
Since $H_1$ stabilizes both $N$ and $U$, it follows that $BH_1 = B = H_1B$.

We claim that $B$ is $H_1$--almost invariant.  
Since $H_1B=B$, it suffices to show that $\delta B$ is $H_1$--finite.
Let $E$ be an edge in $\delta B$, so $E$ has endpoints $b \in B$ and $bs \notin B$, where $s$ is a generator of $G$.
Then $b \in bN$ and $bN \subseteq N_k(U)$, while $bs \in bsN$, with $bsN \subseteq N_k(U^*)$.
Thus $b \in N_k(U)$ and $bs \in N_k(U^*)$, so we must have $E \subseteq N_{k+1}(N)$.

As $d_{Haus}(N, H)<\infty$ and $d_{Haus}(H, H_1)<\infty$ since $[H:H_1]<\infty$, it follows that $\delta B$ is contained in a uniform neighborhood of $H_1$.
Hence $\delta B$ is $H_1$--finite, so $B$ is $H_1$--almost invariant.

It remains to prove that $B$ is nontrivial.
As $U$ is a deep component of the complement of $N$, there is a sequence of vertices $u_1, u_2, \ldots$ in $U$ such that $d(u_i, N) \to \infty$ as $i \to \infty$, and hence $\{ u_i\}$ is $H_1$-infinite.
We can assume that $u_i \notin N_k(U^*)$ for all $i$, thus $u_iN \nsubseteq N_k(U^*)$ for each $i$, so $\{ u_i\}  \subseteq B$, and hence $B$ is $H_1$--infinite.

Since $N$ is 2--separating, $U^*$ contains a deep component of the complement of $N$, and an analogous argument shows that the complement of $B$ is not $H_1$--finite.
Thus $B$ is a nontrivial $H_1$--almost invariant set such that $BH_1 = B$, as desired.
\end{proof}

Finally, Theorem 3.4 of \cite{DunSwen} is the following:

\begin{thm}{\em {\bf \cite{DunSwen}}}\label{DunSwen_thm}
Let $G$ be a finitely generated group and let $H_1$ be a finitely generated subgroup of $G$.
Suppose that, for every infinite index subgroup $K$ of $H_1$, $e(G,K)=1$.
If $G$ contains a nontrivial $H_1$--almost invariant set $B$ such that $BH_1=B$, then $G$ splits over a subgroup commensurable with $H_1$.
\end{thm}

Theorem \ref{splitting_thm0} is immediate from this, Corollary \ref{spl_cor} and Lemma \ref{spl_lem2}.

By combining Theorem \ref{splitting_thm0} with the observation made in Remark \ref{comm_implies_fin_H_dist}, we can see that we could have originally chosen $H$ in Theorem \ref{subgroup_thm} so that $G$ splits over $H$:

\begin{thm}\label{splitting_thm}
Let $G$ be a finitely generated group and let $Y$ be a connected subset of $\mathscr{C}(G)$.
If $Y$ satisfies the deep, shallow, 3--separating and noncrossing conditions, then $G$ contains a finitely generated subgroup $H$ such that 
$$d_{Haus}(Y,H)<\infty$$
and such that $G$ has a splitting over $H$ which has three coends.
\end{thm}

Lastly, the following is an immediate corollary to Theorems \ref{main_thm} and \ref{splitting_thm}.

\begin{cor}\label{splitting_cor}
Let $G$ and $G'$ be finitely generated groups and let $f\co \mathscr{C}(G) \to \mathscr{C}(G')$ be a quasi-isometry.
Suppose that $H$ is a finitely generated subgroup of $G$ such that for any infinite index subgroup $K$ of $H$, $\tilde{e}(G,K)=1$.
Suppose also that $G$ admits a splitting over $H$ that has three coends.

If sufficiently large uniform neighborhoods of $f(H)$ in $\mathscr{C}(G')$ satisfy the noncrossing condition, then $G'$ contains a finitely generated subgroup $H'$ such that $H'$ is quasi-isometric to $H$, $d_{Haus}(H', f(H))<\infty,$ and $G'$ admits a splitting over $H'$ that has three coends.
\end{cor}

\bibliographystyle{alpha}
\bibliography{my_biblio}

\end{document}